\documentclass[10pt]{article}
\usepackage{amssymb}
\usepackage{amsfonts}
\usepackage{graphicx}
\usepackage{amsmath}
\usepackage{tikz}
\usepackage{enumerate}
\usepackage{hyperref}
\newtheorem{theorem}{Theorem}[section]

\newtheorem{assumption}[theorem]{Assumption}

\newtheorem{lemma}[theorem]{Lemma}

\newtheorem{proposition}[theorem]{Proposition}
\newtheorem{remark}[theorem]{Remark}

\newenvironment{proof1}[1][\bf Proof of the Theorem \ref{mains1}]{\noindent\textit{#1.} }{\hfill \rule{0.5em}{0.5em}}
\newenvironment{proof2}[1][\bf Proof of the Theorem \ref{mains2}]{\noindent\textit{#1.} }{\hfill \rule{0.5em}{0.5em}}

\date{}

\usepackage{vmargin}            
\setmarginsrb{2cm}{1cm}{2cm}{1cm}{0cm}{0cm}{2cm}{2cm}
\renewcommand{\eqref}[1]{(\ref{#1})}
\newcommand{\eps}{\varepsilon}

\newcommand{\C}{\mathbb{C}}

\newcommand{\bbw}{\pmb{w}}

\newcommand{\bU}{\pmb{U}}

\newcommand{\bW}{\pmb{W}}

\newcommand{\R}{\mathbb{R}}

\newcommand{\dt}{\frac{d}{dt}}
\newenvironment{proof}[1][Proof]{\noindent\textit{#1.} }{\hfill \rule{0.5em}{0.5em}}

\begin{document}

\title{ Global behavior of N competing species with strong diffusion: diffusion leads to exclusion}

\author{ \textsc{Fran\c cois Castella$^{a}$, \textsc{Sten Madec$^{b}$}} and \textsc{Yvan Lagadeuc$^{c}$} \\
[1mm] $^{a}${\small  \textit{Univ. Rennes 1, UMR CNRS 6625 Irmar, Campus de Beaulieu, 35042 Rennes cedex, France   }   }\\
$^{b}${\small \textit{Univ. Tours,  UMR 7350 LMPT, F-37200 Tours - France}}\\
$^{c}${\small \textit{Univ. Rennes 1, UMR CNRS 6553 Ecobio et IFR/FR Caren, Campus de Beaulieu, 35042 Rennes Cedex, France} }
}
\maketitle

\begin{abstract}
It is known that the competitive exclusion principle holds for a large kind of models involving several species competing for 
a single resource in an homogeneous environment. Various  works indicate that the coexistence is possible in an heterogeneous environment.
We propose a spatially heterogeneous system modeling the competition of several species for a single resource. If  spatial movements are fast enough, 
we show that our system can be well approximated by a spatially homogeneous  system, called aggregated model, which can be explicitly computed. 
Moreover, we show that if the competitive exclusion principle holds for the aggregated model, it holds  for the spatially heterogeneous model too.

\vspace{0.2in}\noindent \textbf{Key words}  Reaction-diffusion systems, Heterogeneous environment, Global behavior, Chemostat

\vspace{0.1in}\noindent \textbf{2010 Mathematics Subject Classification} 35J55, 35J61, 92D40. 

\end{abstract}



\section{Introduction}

In this paper, we are interested in a reaction-diffusion system in a smooth domain  $\Omega\subset \R^p$ modeling the interaction of $N$ species competing for a single resource in a heterogeneous environment
\begin{equation}\label{introeq0}
\left\{ \begin{array}{lll}
\frac{\partial}{\partial t} R^\eps=I-\sum_{i=1}^N \frac{1}{\lambda_i}f_i(R^\eps)V_i^\eps-m_0 R^\eps+\frac{1}{\eps} A_0R^\eps&& \text{ on $\Omega$}\\
\frac{\partial}{\partial t} V_i^\eps=(f_i(x,R^\eps)-m_i(x))V_i^\eps+\frac{1}{\eps} A_i V_i^\eps&i=1..N&\text{ on $\Omega$}\\
\partial_n R^\eps=0&&\text{ on $\partial \Omega$}\\
\partial_n V_i^\eps=0&i=1,..,N& \text{ on $\partial \Omega$}\\
R^\eps(t=0)=R^0\geq 0&&\\
V_i^\eps(t=0)=V_i^0\geq0 &i=1,\cdots,N&\\
\end{array}\right.\end{equation}
where, for $i=0,\cdots,N$,  $A_i=div(a_i(x) \nabla \cdot)$, with $a_i\in C^1(\overline{\Omega})$ is positive, 
and  $\partial_n=\nabla\cdot \vec{n}$ denotes the normal derivative on $\partial\Omega$, and, at any position $x\in\Omega$ and instant $t\geq 0$,\\
\begin{itemize}
\item $R^\eps(x,t)$ is the concentration of resource,
\item $I(x)\geq 0$ is the input of substrate,
\item $m_0(x)>0$ is a natural decreasing factor modeling phenomena as sedimentation and dilution,
\item $V_i^\eps(x,t)$ is the concentration of the species $i$,
\item $f_i(R)(x,t)=f_i(x,R(x,t))$ is the consumption rates of the species $i$ on the resource $R$,
\item $\lambda_i\in (0,+\infty)$ is the growth yield of the species $i$,
 \item $m_i(x)>0$ is the mortality rates of the species  $i$,
 \item $\frac{1}{\eps}\in(0,+\infty)$ is the common diffusion rate.
 \end{itemize}
 \paragraph{}
\noindent 
The resource is the only limiting factor in this model and species interact indirectly  through their respective  consumption of the resource. 
Without spatial structure, this model is known as the well stirred  chemostat which has received considerable attention
\cite{LI, SaMa,waltbook, WL, WX}. In the well stirred chemostat it is known that generically, 
all (nonnegative) steady states are of the form $(r,u_1,\cdots,u_n)$ where at most one $u_i$ is positive and exactly one of these steady states is stable. 
Under some additional assumptions on the parameters  this only stable steady state is a global attractor. 
In other words, the  competitive exclusion principle (CEP) holds: at most one species survives as $t\to+\infty$. 
In this perspective, our  model  is motivated by the following question. {\it Can the spatial heterogeneity permits the long term coexistence of many species.}\\
The influence of spatial heterogeneity in population dynamics has received considerable attention. 
We refer to the review of Lou \cite{Loureview} and references therein. 
Most of the time,   spatial heterogeneity is considered in prey-predator system or Lotka-Volterra competing system. 
There is very few consideration of spatial heterogeneity in systems of species competing for a single resource.\\ 

Waltman {\it et al.} \cite{walt1,SoW} studied this kind of system for two species in one spatial dimension with $A_i=\partial_{xx}$ 
for  $i=0,1,2$ and $m_i\equiv 0$, $I\equiv0$ with Michaelis-Menten consumption rates {\it independent on $x$} and Robin boundary conditions. 
Wu \cite{Wu1} generalized this system in any spatial dimensions and showed the existence of positive stationary solution for two species. 
Recently  Nie and Wu \cite{NieWu} show  uniqueness and global stability properties for this stationary solution under some technical assumptions. \\ 
The above mentioned works use strongly a monotone method  which holds only for {\it two species} and under the additional condition that both the diffusion rates $a_i$ {\it do not depend on $i$}.
The other cases has been very little studied. Waltman {\it et al.} \cite{walt4} treat the case of two species  and different but close enough diffusion rates, 
by using a perturbation method.  For  more than two species,
Baxley and Robinson \cite{BaxRob} show the existence of a stationary solution {\it near a 
bifurcation point} for general elliptic operators $A_i-m_i$ and Michaelis-Menten type consumption functions.\\

Our system is slightly different from the above cited works  since here, the spatial heterogeneity takes place directly 
on the reaction terms rather than on the boundary conditions. If a similar analysis can be done for two species 
in the case of operators $A_i-m_i$ which do not depend on $i$, this different formulation allows us to take  Neumann boundary conditions. 
This make possible to investigate phenomena occurring when the diffusion rates $\frac{1}{\eps}$ varies, in a situation wherethe operator
$A_i-m_i$ are  {\it species dependent}.
Stationary solution of this  system
for two species and any diffusion rates has been investigated by Castella and Madec in \cite{CasMad} using global bifurcation methods.
For any number of species,   the stationary solutions has been studied by Ducrot and Madec in \cite{DUMA} when the diffusion rates $\frac{1}{\eps}$ tends to $0$.
The present paper focuses on the opposite case $\frac{1}{\eps}\to +\infty$ and investigates both the stationary solutions and the global dynamic.

\paragraph{}
The purpose of this article is to show that the dynamics of the system  is well described by the dynamics of an associated averaged system, 
called aggregated system, if the diffusion rate is large enough. 
In particular, we show that if the  CEP holds for the aggregated problem, then the CEP holds for  the original problem  for small enough $\eps$. 
Note that the model of homogeneous chemostat is based on the assumption  that the chemostat is well mixed. 
This  study makes the validity of this assumption more precise and clarifies the  parameters of the associated homogeneous problem.\\


\paragraph{}
\noindent Here, we  investigate a fast migration problem:
\begin{equation}\label{introeq1}
\frac{d}{dt}\bW^\eps(x,t)=\mathcal{F}(x,\bW^\eps(x,t))+\frac{1}{\eps} K \bW^\eps(x,t)
\end{equation}
 where $\bW^\eps(t):=\bW^\eps(\cdot,t)$ is a vector with $N+1$ components both belonging  to a well chosen Banach space. 
 The demography is described by the reaction terms $\mathcal{F}(\bW^\eps)$  and the operator   $K$ models the spatial movements. 
Such a complex system, involving  $N+1$ partial derivatives equations, appears naturally when one considers phenomena acting on different time scales.
It is well known (see for instance, Conway, Hoff and Smoller \cite{CHS}, Hale and Carvalho \cite{Hale} and references therein) 
that  systems like \eqref{introeq1} are well 
described, with an $O(\eps)$ error term, by the averaged system

\begin{equation}\label{introeq2}
\frac{d}{dt}\bbw^\eps(t)=
\frac{1}{|\Omega|}\int_\Omega\mathcal{F}\left(x,\bbw^\eps(t)\right)dx
\text{ where } \bbw^\eps(t)=\frac{d}{dt}\frac{1}{|\Omega|}\int_\Omega \bW^\eps(x,t)dx
\end{equation}
as soon $\eps$ is small enough. In fact, in the case of {\it homogeneous} reaction-terms, the asymptotic profiles are given {\it exactly} by the system
\eqref{introeq2}, while for {\it spatially dependent} reaction-terms, the $O(\eps)$ error term remains.\\
Hence, we use here an alternative approach using the invariant manifold theory (see \cite{Ca}) which provides precise estimates 
on the error between \eqref{introeq2}
and \eqref{introeq1}. These estimates are useful to describe {\it exactly } the long time dynamic of \eqref{introeq1} for small enough $\eps$.

Basically,  the central manifold theorem allows to reduce the study of \eqref{introeq1} to this of the aggregated system  \eqref{introeq2}
involving only $N+1$ differential equations. Many authors use this approach in populations dynamics. We refer to  Poggiale, 
Auger and  Sanchez  \cite{AP, po3, po1} for results on this subject in differential systems, Arino {\it et al} \cite{ASBA} for age-structured model  
and most recently,  Castella \textit{et al.} \cite{Castella2009} and Sanchez {\it et al.} \cite{po5} in  problems involving functional space.      

\paragraph{}
The essential features for this approach to be valid is that the solution space $\mathcal{H}$ 
admits a decomposition on the form $\mathcal{H}=E\oplus F$ where $E=ker(K)$ and $F$ is invariant under $K$ 
while the real part of the spectrum of $K_{|F}$   belongs to $(-\infty,-\alpha)$ for some $\alpha>0$. Note that such is  the case for $\Delta$ 
with zero flux boundary conditions.
\noindent Under this conditions, projecting the system $S_\eps$ on $E$ and $F$ and denoting $X^\eps$ and $Y^\eps$ the projections of $\bW^\eps$ on $E$ and $F$ 
respectivly, leads to  the following ``slow-fast'' system
 \begin{equation}\label{pbgallentrap}
 \left\{\begin{array}{l}
 \partial_t X^\eps(t)=\mathcal{F}_0(X^\eps(t),Y^\eps(t))\\
 \partial_t Y^\eps(t)=\mathcal{G}_1(X^\eps(t),Y^\eps(t))+\frac{1}{\eps} K Y^\eps(t).\\
 \end{array}\right..
 \end{equation}
Here, $X^\eps\in E$ is the slow variable  and $Y^\eps \in F$ is the fast variable.\\

In essence, the central manifold theorem  asserts the existence of an invariant manifold 
$\mathcal{M}^\eps=(X^\eps,h(X^\eps,\eps))\in E\times F$ verifying $h(X^\eps,\eps)=O(\eps)$ as $\eps\to 0$ and attracting exponentially fast any trajectories.
Thus, the complex dynamics of $S_\eps$  may be approach,  up to exponentially small error term, by the dynamics reduced to $\mathcal{M}^\eps$, 
which is described by only $N+1$ differential equations rather than $N+1$ partial differential equations. This reduced system reads shortly
\begin{equation}\label{reduced1intro}\left\{\begin{array}{l}
\frac{d}{dt} X^\eps(t)=\mathcal{F}_0(X^\eps,h(X^\eps,\eps))\\
Y^\eps(t)=h(X^\eps,\eps),\\
\end{array}.\right.
\end{equation}

\noindent Generaly, the central manifold $\mathcal{M}^\eps$ can not be explicitly computed. 
Explicit approximations of  $h(x,\eps)$ can though be computed at any order $\eps^l$.
This allows to describe the dynamic of the reduced system up to an additional polynomial small error term of order $\eps^{l+1}$.
In this works, we concentrate our study on the order $0$ reduced system,  called the aggregated system, which reads\\
\begin{equation}\label{reduced3intro}
\frac{d}{dt} X^{\eps,[0]}(t)=\mathcal{F}_0(X^{\eps,[0]},0),\; Y^{\eps,[0]}(t)=h(X^{\eps,[0]}(t),\eps).
\end{equation}

Explicit calculation shows that  the system \eqref{reduced3intro} is a simple homogeneous chemostat system. 
It follows that   long time behavior of its solutions is completely known for a large choice of function $\mathcal{F}_0$. 
The aim of this work is to transfer qualitative properties of \eqref{reduced3intro} to the original system $S_\eps$.

\paragraph{}
This article is organized as follow.
In the second section, we precise the assumptions on the model and we state a theorem assuring the existence and uniqueness of 
 classical solutions which are uniformly bounded independently on $t$ and $\eps$. We then restate the system on a slow-fast 
 form allowing to apply the central manifold theorem.  In the end of the second section, we state our two main results: 
Theorems \ref{mains1} and \ref{mains2}.  In the third section, we beging to state the central manifold Theorem \ref{centralmanifold}
and a Theorem describing the exponential convergence towards the central manifold \ref{errorbound}. 
Next, we use these two Theorems to prove several general results on slow-fast systems.
In the fourth section, we use these general results to prove the Theorems \ref{mains1} and \ref{mains2}.
The main result (Theorem \ref{mains2}) states that, if the CEP holds for the aggregated system, then   it holds for the original system too, for small 
enough $\eps$. Hence, only one species can win the competition, namely the best  competitor in average. 
This best competitor in average  can be explicitly computed. In the fith section, we discuss through some examples three important phenomena
determining which species is the best  competitor in averaged. These phenomena give good informations on how
a heterogeneous environment may promote the coexistence for an intermediate diffusion rate. The sixth section concludes the paper.

\section{Model and main results}

\subsection{The model}
First, by denoting $U_i^\eps(x,t)=\lambda_i^{-1}V_i^{\eps}(x,t)$ we see that $(R^\eps,V_1^\eps,\cdots,V_N^\eps)(x,t)$ 
is a solution of the system \eqref{introeq0} if and only if 
$(R^\eps,U_1^\eps,\cdots,U_N^\eps)(x,t)$ is a solution of
\begin{equation*}
 \begin{array}{lc}
S_\eps:\;\left\{ \begin{array}{lll}
{\displaystyle \dt R^\eps(x,t)=I(x)-\sum_{i=1}^N f_i(x,R^\eps(x,t))U_i^\eps-m_0(x) R^\eps(x,t)+\frac{1}{\eps} A_0R^\eps(x,t)}&& \text{ on $\Omega$}\\
{\displaystyle\dt U_i^\eps(x,t)=(f_i(x,R^\eps(x,t))-m_i(x))U_i^\eps(x,t)+\frac{1}{\eps} A_i U_i^\eps(x,t)}&i=1,\cdots,N&\text{ on $\Omega$}\\
{\displaystyle\partial_n R^\eps(x,t)=0}&&\text{ on $\partial \Omega$}\\
{\displaystyle\partial_n U_i^\eps(x,t)=0}&i=1,\cdots,N& \text{ on $\partial \Omega$}\\
{\displaystyle R^\eps(x,0)\geq 0}&&\\
{\displaystyle U_i^\eps(x,0)\geq0 }&i=1,\cdots,N&.\\
\end{array}\right.\end{array}\end{equation*}
This system  can be shortly written  as
\begin{equation}\label{PBshort}\left\{\begin{array}{lc}
\dt \bW^\eps(x,t)=\mathcal{F}(x,\bW^\eps(x,t))+\frac{1}{\eps}K\bW^\eps(x,t)& \text{ $t>0$ et $x\in \Omega$, }\\
\partial_n(\bW^\eps)(x,t)=0, &\text{ $t>0$ et $x\in\partial\Omega$}\\
\bW^\eps(x,0)=\left(R^0(x),U_1^0(x),\cdots,U_N^0(x)\right),& \text{ $x\in \Omega$}\end{array}\right.
\end{equation}
where
\begin{itemize}
\item $\bW^\eps(x,t)=(R^\eps(x,t),U_1^\eps(x,t),..,U_n^\eps(x,t))^T$, 
\item $\mathcal{F}(x,\bW^\eps(x,t))=\left(\begin{array}{c}
I(x)-m_0(x) R^\eps(x,t)-\displaystyle\sum_{i=1}^N U_i^\eps(x,t) f_i(x,R^\eps(x,t))\\                                                                           
\big(f_1(x,R^\eps(x,t))-m_1(x)\big)U_1^\eps(x,t)\\\vdots\\\big(f_N(x,R^\eps(x,t))-m_N(x)\big)U_N^\eps(x,t)      
\end{array}\right)$, 
\item $K=diag(A_i)$.
\end{itemize} 
In the sequel,  the {\it same} symbol $\mathcal{F}$ is used  to refer to the Nemitski operator $\bW\mapsto \mathcal{F}(\bW)$ where 
$$\mathcal{F}(\bW)(x)=\mathcal{F}(x,\bW(x)).$$
\begin{remark}\label{general}
All the results of this works hold true for any uniform elliptic operators  $A_i$, or integral operators verifying some property (see \cite{Castella2009}). 
One can also investigate gradostat-like models by taking $\Omega=\{1,\cdots,P\}$ and $A_i\in \R^{P\times P}$ an
irreducible matrix with nonnegative off diagonal entries such that the sum of each column is  $0$. The results proved here hold as well in this case.
\end{remark}

\noindent In the sequel, we make the two following assumptions insuring that the system $S_\eps$ admits an 
unique global classical positive solution which is uniformly  bounded in $C^0\left(\overline{\Omega}\right)$.
\begin{assumption}[Assumption on the parameters]\label{chap5hyp0}
$\quad$
\begin{itemize}
\item $I\in C^1(\overline{\Omega},\R^+)$ and  $I\not\equiv 0$.
\item For $i=0,\cdots,N$, $m_i\in C^1(\overline{\Omega})$ and $m_i(x)>0$.
\item For $i=0,\cdots,N$, $a_i\in C^1(\overline{\Omega})$ and for all $x\in \overline{\Omega},$ we have $a_i(x)>0$.
\end{itemize}
\end{assumption}
The assumption
\noindent $I\not\equiv0$ means that there is always an input of resource in the system. 
If $I\equiv 0$, then $(0,\cdots,0)\in \R^{N+1}$ is a global attractor and the problem is trivial.
\begin{assumption}[Assumptions on the consumption functions]\label{chap5hypf}
For each $i=1,\cdots,N$, we assume
\begin{itemize}
\item  $\forall R\in \R^+$, $f_i(\cdot,R):x\mapsto f_i(x,R)$ belongs to $C^1(\overline{\Omega})$ and take values in $\R^+$, 
\item $\forall x\in\Omega$, $f_i(x,\cdot):R\mapsto f_i(x,R)$  belongs to $C^1(\R^+)$ and is increasing. 
Moreover, $R\mapsto D_Rf_i(x,R)$ is locally Lipschitz.
\item  $\forall x\in\Omega$, $f_i(x,0)=0$.  
\end{itemize}
\end{assumption}
\begin{remark} The monotonicity of $R\mapsto f_i(x,R)$ is not fundamental in our analysis. 
Indeed, our results hold true if $\int_\Omega f_i(x,r)dx=\int_\Omega m_i(x)dx$ has at most one solution $r_i^*$ and 
if the conclusions of the proposition \ref{hyp0} are verified. However, in order to avoid technical difficulties,
we restrict ourself to the case of increasing consumption functions.\end{remark}

\noindent It is classical that the system $S_\eps$ conserves the positive quadrant and admits an unique solution for a time $\tau$ small enough. 
Moreover, the maximum principle implies that $R^\eps$ verifies for any $t>0$ the uniform bound $\|R^\eps(\cdot,t)\|\leq M$ for some $M>0$ independent of the time $t$.  
It follows, using  standard results on parabolic systems (see \cite{Henry, Pazy}), that the solution is well defined and classical globally in time. 
Finally, it can be proven by a  $L^p$ estimates method\footnote{The key to apply this method is as follows. 
1. There is a $L^1$ control on the solutions uniformly in time $\|\bW^\eps(\cdot,t)\|_1\leq C$.
2. The system has a particular structure. For our system, the system is triangular since the $U_i^\eps$ are coupled indirectly through $R^\eps$.
3. There is a uniform bound for a (well chosen) component of $\bW^\eps$. Here, $\|R^\eps(\cdot,t)\|\leq M$.
}  
(Hollis {\it et al.} \cite{Pierre}),  
that the system $S_\eps$ admits a unique classical positive solution which is uniformly bounded in time in  $\left(C^{0}(\overline{\Omega})\right)^{N+1}$. 
More precisely, 
the following theorem\footnote{The theorem \ref{chap5thmaj} holds true with an initial condition 
$W^\eps(0)\in (L^\infty({\Omega}))^{N+1}$. However, since $W^\eps(t)$  belongs to $\left(C^0(\overline{\Omega})\right)^{N+1}$ for any $t>0$, one reduce ourself 
to the case of continuous initial data. This will simplify the statement of the main results. Finally, the solution is more regular since 
$\bW^\eps \in C^1((0,+\infty), W^{2,p})$ for any $p>1$.} holds (see \cite{thesis} chapter III for this specific case).
\begin{theorem}\label{chap5thmaj}
Assume that $W^\eps(0)\in (C^{0}(\overline{\Omega}))^{N+1}$ is nonnegative. For each  $\eps>0$,  the system $S_\eps$ admits an unique  solution 
$\bW^\eps=(R^\eps,U_1^\eps,..,U_N^\eps)\in C^1\left(]0,+\infty[;(C^{0}(\overline{\Omega}))^{N+1}\right)$ which is nonnegative.
Moreover,  for each $\eps_0>0$ and $\eps\in(0,\eps_0)$, there exists a constant $M(\eps_0)$ independent on $t$ and  $\eps$ such that
$$\|R^\eps(\cdot,t)\|_\infty+\sum_{i=1}^N\|U^\eps_i(\cdot,t)\|_\infty \leq M(\eps_0).$$
\end{theorem} 
Armed with this Theorem, we are in position to analyze the asymptotic behavior
of the dynamic of $S_\eps$ as $\eps\to 0$.\\

\subsection{Slow Fast Form}\label{sec22}

When seen as an operator on $L^2(\Omega)$, the operator $A_i^2:=div(a_i(x)\nabla\cdot)$ with homogeneous Neumann boundary conditions is defined as 
$$D(A_i^2):=\left\{ U\in H^1(\Omega)\; \exists V\in L^2(\Omega),\;\forall \phi\in H^1(\Omega),\; \int_\Omega a_i(x)\nabla U(x) \nabla \phi(x)dx
=-\int_\Omega V(x) \phi(x)\right\}.$$
$$A_i^2 U:=V,\; \forall U\in D(A_i^2).$$

In order to obtain  uniform estimates, we prefer to focus on the operator $A_i^\infty:=div(a_i(x)\nabla\cdot)$ 
when acting on the set of continuous function $(C^0(\overline{\Omega}),\|\cdot \|_\infty)$ where $\|f\|_\infty=\sup_{x\in \overline{\Omega}} (|f(x)|)$. 
Hence, we define
$$D(A_i^\infty)=\left\{ U\in D(A_i^2)\cap C^0(\overline{\Omega}),\; A_i^2 U\in C^0(\overline{\Omega})\right\},$$
$$A_i^\infty U = A_i^2 U,\;\forall U\in D(A_i^\infty) $$ 

\noindent We have
$$ker(A_i^\infty)=span(1)=\R\text{ and }\widetilde{F}:=Im(A_i^\infty)=\left\{U\in C^0(\overline{\Omega}),\; \int_\Omega U=0\right\}.$$
One gets clearly $C^0(\overline{\Omega})=ker(A_i^\infty)\oplus Im(A_i^\infty)$.
\noindent Now, we define  the Banach space
$\left(C^0(\overline{\Omega})\right)^{N+1}$
together with the norm
$$\|(U_0,\cdots,U_N)\|_\infty=\sum_{i=0}^N \|U_i\|_\infty.$$
and  the operator $K^\infty=diag(A_i^\infty)$ acting on $\left(C^0(\overline{\Omega})\right)^{N+1}$.
The Kernel and the range of $K^\infty$ are respectively\footnote{In the case of most general operator (see remark \ref{general}), 
one has $ker(A_i^\infty)=span(\phi_i)$ for some positive function $\phi_i$ and $\widetilde{F}_i=ker(A_i)^{\bot}$.
For the sake of simplicity we reduce ourself to the case of operator $A_i^\infty$ s.t. $\phi_i=1$ and $\widetilde{F}_i=span (1)^\perp$ do not depends on $i$.} 

$$E:=ker(K^\infty)=\R^{N+1} \text{ and }F:=Im(K^\infty)=\widetilde{F}^{N+1}.$$
The spaces $E$ and $F$ are cleary two complete subspaces of $\left(C^0(\overline{\Omega})\right)^{N+1}$ and 
one has
$$\left(C^0(\overline{\Omega})\right)^{N+1}=E\oplus F.$$
The projections of $\left(C^0(\overline{\Omega})\right)^{N+1}$ on $E$ and $F$, denoted by $\Pi_E$ and $\Pi_F$ respectively, are given   explicitly    by
$$\Pi_E(V_0,\cdots,V_N)=\frac{1}{|\Omega|}\left(\int_\Omega V_0,\cdots,\int_\Omega V_N\right)\text{ and }\Pi_F=I_d-\Pi_E.$$ 
The restrictions of the norm $\|\cdot \|_\infty$ on   $E$ and $F$ are noted respectivly
$$\|(u_0,\ldots,u_N)\|_E=\sum_{i=0}^N |u_i|,\qquad \|(U_0,\ldots,U_N)\|_F=\sum_{i=0}^N\|U_i\|_\infty.$$
Finally, let us define the norm $\|\cdot \|_{E\times F}$ on the Banach space $E\times F$ by
$$\forall (u,V)\in E\times F,\quad \|(u,V) \|_{E\times F}=\|u\|_E+\|V\|_F.$$

\noindent One verifies easily that the map $E\times F\to \left(C^0(\overline{\Omega})\right)^{N+1}=E\oplus F\; : \;(u,v)\mapsto u+v$ defines an isomorphism between
the banach spaces ($E\times F, \|\cdot\|_{E\times F})$ and $\left(\left(C^0(\overline{\Omega})\right)^{N+1},\; \|\cdot\|_\infty\right)$.
Thus, it is equivalent to obtain estimates on $E\times F$ and on $\left(C^0(\overline{\Omega})\right)^{N+1}$.\\

\noindent The above considerations permits to restate the system $S_\eps$ on an equivalent ``slow-fast'' form by projecting $S_\eps$ on $E$ and $F$ respectivly.
Let $\bW^\eps(t)$ be a solution of $S_\eps$. The slow variable   $X^\eps:=\Pi_E(\bW^\eps)\in E$ is the vector of the mean mass of  resource and species. More  precisely,
$$X^\eps=\left(\frac{1}{|\Omega|}\int_\Omega R^\eps,\frac{1}{|\Omega|}\int_\Omega U_1^\eps,..,\frac{1}{|\Omega|}\int_\Omega U_N^\eps\right)\in\R^{N+1}.$$  
The fast variable is simply $Y^\eps:=\Pi_F \bW^\eps=\bW^\eps-X^\eps\in F$. \\
Furthermore, thanks to the boundary conditions, we have  $\Pi_E (K^\infty\bW^\eps)=0$ and $\Pi_F(K^\infty\bW^\eps)=K^\infty\Pi_F\bW^\eps=K Y^\eps$ 
where we have note $K:=K^\infty_{|F}$ the restriction of $K^\infty$ to $F$. \\ 
Projecting the system $S_\eps$ on $E$ and $F$ yields to the equivalent system

\begin{equation*}\label{fastslow}
	\begin{array}{ll}
			\left(S_\eps^{sf}\right)\;:\; &
							\left\{\begin{array}{l}
 								\frac{d}{dt} X^\eps(t)=\mathcal{F}_0(X^\eps,Y^\eps)\\
								\frac{d}{dt} Y^\eps(t)=\mathcal{G}_1(X^\eps,Y^\eps)+\frac{1}{\eps}KY^\eps\\
								\partial_n X^\eps=0\\
								\partial_n Y^\eps=0\\
								X^\eps(0)=\Pi_E(\bW(0))\\
								Y^\eps(0)=\Pi_F(\bW(0))
							\end{array}\right.
	\end{array}
\end{equation*}
where $\mathcal{F}_0(X^\eps,Y^\eps)=\Pi_E\mathcal{F}(X^\eps+Y^\eps)$ and $\mathcal{G}_1(X^\eps,Y^\eps)=\mathcal{F}(X^\eps+Y^\eps)-\mathcal{F}_0(X^\eps,Y^\eps)$.\\

\noindent In its slow-fast form, the system describes on the one hand the {\it slow } dynamics on the 
kernel $E$ of $K^\infty$, and on the other hand the {\it fast }dynamics 
on the orthogonal $F$ of $E$. These two dynamics are coupled which results in  complex dynamics of $S_\eps^{sf}$. 
However, this complex dynamics may be completly understood using   the central manifold theory.\\
 
Basically (see section \ref{seccentralmanifold} for a precise statement), this theory asserts that there exists a manifold 
$\mathcal{M}^\eps=\{(x,h(x,\eps)),\;x\in E\}\in E\times F$
which is invariant for $S_\eps^{sf}$. It verifies moreover  $h(x^\eps,\eps)=O(\eps)$ and $\mathcal{M}^\eps$ attracts any trajectory 
 exponentially fast in time. 
 The system on $\mathcal{M}^\eps$ reads
\begin{equation}\label{fsap1}
\left(S_\eps^{[\infty]}\right),\quad\frac{d}{dt} X^{\eps,[\infty]}(t)=\mathcal{F}_0(X^{\eps,[\infty]}(t),h(X^{\eps,[\infty]}(t),\eps)),
\, Y^{\eps,[\infty]}(t)=h(X^{\eps,[\infty]}(t),\eps).
\end{equation}
Since $h(x^\eps,\eps)=O(\eps)$  as $\eps\to0$,  one obtains the following  system, as a first approximation.

\begin{equation}\label{fsap0}
\left(S_\eps^{[0]}\right),\quad\frac{d}{dt} X^{[0]}(t)=\mathcal{F}_0(X^{[0]}(t),0), \, Y^{\eps,[0]}(t)=h(X^{\eps,[0]}(t),\eps).
\end{equation}
An important fact in the sequel is that the dynamic of $S_\eps^{[\infty]}$ is completely determined by its first equation: the following {\it O.D.E} system
\begin{equation}
\left(S_\eps^{c}\right),\quad\frac{d}{dt} X^{\eps,[\infty]}(t)=\mathcal{F}_0(X^{\eps,[\infty]}(t),h(X^{\eps,[\infty]}(t),\eps))
\end{equation}
In many cases, $S_\eps^c$ can be seen as a {\it regular perturbation} of the first equation of $S_\eps^{[0]}$, that is
\begin{equation}
\left(S_0^{c}\right),\quad\frac{d}{dt} X^{[0]}(t)=\mathcal{F}_0(X^{[0]}(t),0)
\end{equation}

\subsection{Main results}
The general strategy to prove our results is as follow.\\
When $S_\eps^c$ can be seen as a regular perturbation of $S_0^c$, many properties of $S_0^c$ can be transfer
to $S_\eps^c$ which infers properties of $S_\eps^{[\infty]}$. The system $S_\eps^{[\infty]}$ is exaclty the slow-fast system $S_\eps^{sf}$ reduced to
the invariant  manifold $\mathcal{M}_\eps$. Since  $\mathcal{M}_\eps$ attracts exponentially fast in time 
any trajectectory  of $S_\eps^{[sf]}$, many properties of $S_\eps^{[\infty]}$ yield properties for $S_\eps^{[sf]}$ which is
equivalent to the original system $S_\eps$. This strategy may be summarized as follow.
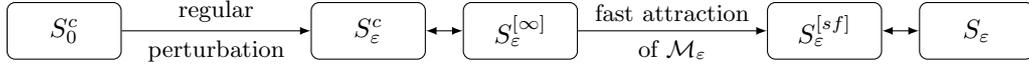
\begin{figure}[h!]
\centering
\begin{tikzpicture}[x=1cm,y=1cm,>=stealth]
\tikzstyle{site}=[draw, rectangle,rounded corners=3pt, minimum width=1.5cm, minimum height=0.75cm]
\node[site] (1)at(-4,2){$S_0^c$};
\node[site] (2)at(0,2){$S_\eps^c$};
\node[site] (3)at(2,2){$S_\eps^{[\infty]}$};
\node[site] (4)at(6,2){$S_\eps^{[sf]}$};
\node[site] (5)at(8,2){$S_\eps$};
\draw[->,>=latex] (1) to node[below]{\small perturbation} node[above]{\small regular}(2) ;
\draw[<->,>=latex] (2) to (3) ;
\draw[->,>=latex] (3) to node[above]{\small fast attraction} node[below]{\small of $\mathcal{M}_\eps$} (4) ;
\draw[<->,>=latex] (4) to (5) ;
\end{tikzpicture}
\end{figure}

\noindent The essential difficulties in the proofs appear in  transfering some properties from $S_\eps^{[\infty]}$ to $S_\eps^{[sf]}$. This part 
uses strongly  theorem \ref{errorbound}. \\

\noindent In order to apply the above mentioned strategy, the first step is to study $S_0^c$.
\noindent In the case of our system, $S_0^{c}$ reads explicitly
\begin{equation}\label{reduced3}
\begin{cases}
\frac{d}{dt} r= \widetilde{I}-\widetilde{m_0}r-\sum\limits_{i=1}^N \widetilde{f_i}(r) u_i ,\\
\frac{d}{dt} u_i=\left(\widetilde{f_i}(r)-\widetilde{m_i}\right)u_i,\;
i=1,\cdots,N.
\end{cases}
\end{equation}

\noindent where $\widetilde{I}=\frac{1}{|\Omega|}\int_\Omega I(x)dx$, 
$\widetilde{m_0}=\frac{1}{|\Omega|}\int_\Omega m_0(x)dx$ and for $1\leq i\leq N$,
$${\displaystyle\widetilde{m_i}=\frac{1}{|\Omega|}\int_\Omega m_i(x)dx\;\text{ and }\;\widetilde{f_i}(r)=\frac{1}{|\Omega|}\int_\Omega f_i(x,r)dx.}$$

One defines $r_0^*=\widetilde{I}/\widetilde{m_0}$. 
For any $i\in \{1,\cdots,N\}$, since $f_i(x,\cdot)$ is increasing,  
$\widetilde{f}_i(\cdot)$ is an increasing function and
one may define the number $r_i^*$ as shown in the figure \ref{figureri}.

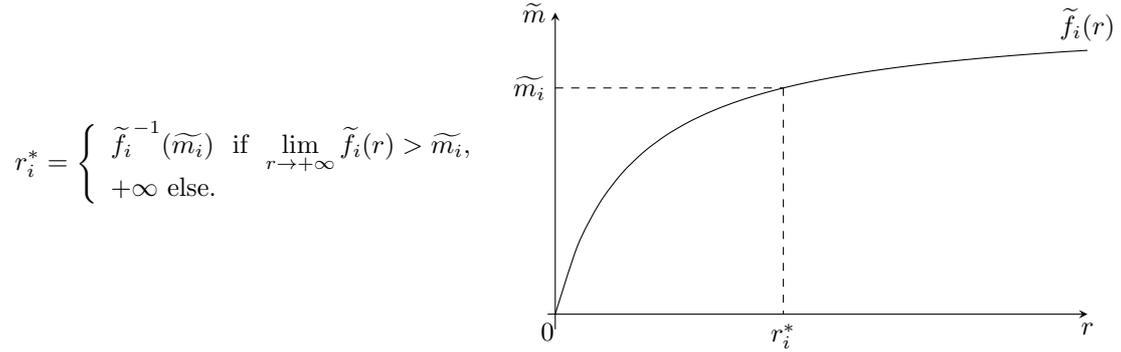
\begin{figure}[!ht]
\centering
 \begin{tikzpicture}[x=1cm,y=2cm,>=stealth]
\draw[->](-0.1,0) -- (7,0) node [below] {$r$};
\draw[->] (0,-0.1) -- (0,2) node [left] {$\widetilde{m}$};
\draw plot [domain=0:7,smooth] (\x,{(2*\x) /(1+\x)}) node [above, sloped] {$\widetilde{f_i}(r)$};
\draw[-,dashed] (0,1.5) node [left] {$\widetilde{m_i}$}-- (3,1.5);
\draw[-,dashed] (3,1.5)--(3,0) node [below] {$r_i^*$};
\draw (-0.1,0) node [below] {$0$};
\draw (-4,1) node  {${\displaystyle r_i^*=\left\{\begin{array}{ll} 
				\widetilde{f_i}^{-1}(\widetilde{m_i})$ \text{ if } 
				${\displaystyle\lim_{r\to +\infty} \widetilde{f_i}(r)>\widetilde{m_i}}, \\
				+\infty \text{ else.}
				\end{array}\right.}$ };
 \end{tikzpicture}
\caption{\label{figureri} Definition of $r_i^*$.}
\end{figure}
\noindent

\noindent The nonnegative stationary solutions of $S_0^c$ are well known  and are described in the following proposition.

\begin{proposition}[Stationnary solutions of the aggregated system $S_0^c$  (see \cite{waltbook})]\label{stationaryagregated}
Under the assumptions \ref{chap5hyp0} and \ref{chap5hypf},we have.
 \begin{itemize}
  \item[(i)] The system $S_0^c$ always admits the stationary solution $p_0^*=(r_0^*,0,\cdots,0)$. 
	     This solution is hyperbolic\footnote{That is, $0$ is not an eigenvalue of $D_X \mathcal{F}_0 (X^{*},0)$.}
 if $r_0^*\neq r_i^*$ for any $i\geq 1$.\\
	     If moreover  $r_0^*< r_i^*$ for all $i\geq 1$. 
	    Then $p_0^*$ is the only nonnegative stationary solution of  $S_0^c$ and is (linearly) asymptotically stable\footnote{An hyperbolic solution $X^{*}$ is say to be (linearly) asymptotically  stable  (resp. unstable) 
if the real part of all the eigenvalue of $D_X \mathcal{F}_0 (X^{*},0)$ is negative 
(resp. if the real part of almost one eigenvalue is positive). In the sequel, we  do not precise (linearly).}.
\item[(ii)] Let $i\in\{1,\cdots,N\}$ and suppose that $r_i^*<r_0^*$ and $r_i^*\neq r_j^*$ for all $j\in\{1,\cdots,N \}\setminus \{i\}$. 
	    Then the system $S_0^c$ has one non-negative stationary solution 
	    $$p_i^*=(r_i^*,0,\cdots,0,u_i^*,0,\cdots,0)\quad 
	    \text{where } u_i^*=\frac{\widetilde{m_0}}{\widetilde{m_i}}\left(r_i^*-r_0^*\right)>0.$$
	    Moreover, this solution is hyperbolic and is asymptoticaly stable if $r_i^*<r_j^*$ for all $j\in\{0,\cdots,N \}\setminus \{i\}$ 
	     and unstable else.
\item[(iii)] Suppose that $r_i^*\neq r_j^*$ for all $i\neq j$ and $r_i^*<r_0^*$ for all $i\geq 1$.  
	     Then the system $S_0^c$ has exactly $N+1$ non-negative stationary solutions:
	     $p_i^*,\; i=0,\cdots,N$. 
	    Moreover, all these solutions are hyperbolic and exactly one of these is stable: 
	    $p_{i_0}^*$ where $r_{i_0}^*=\min\{r_0^*,\cdots,r_N^*\}.$
\end{itemize}
\end{proposition}
The knowledge of the stationary solutions of  $S_0^c$
permits to completely describe  the stationary solutions of  $S_\eps$. This yields our  firth main result, which is proved in section \ref{proofs}.\\

\begin{theorem}[Stationary solutions of the original system $S_\eps$]\label{mains1}
There exist two positive scalars $\eps_0$ and $C$ such that for all $\eps\in (0,\eps_0)$ the following holds.
\begin{itemize}
 \item[(i)] Suppose that $r_0^*< r_i^*$ for all $i\geq 1$. 
	    Then the system $S_\eps$ has only one nonnegative stationary solution 
	    $W_0^\eps(x)=(R_0^\eps(x),0,\cdots,0)$ which is hyperbolic and stable and verifies,
	    $$\|R_0^\eps(\cdot)-r_0^*\|_\infty\leq C \eps.$$
\item[(ii)]Let $i\in\{1,\cdots,N\}$ and suppose that $r_i^*<r_0^*$ and $r_i^*\neq r_j^*$ for all $j\in\{1,\cdots,N \}\setminus \{i\}$.
	   Then the system $S_\eps$ has (at least) one non-negative stationary solution
	   $$W_i^\eps(x)=(R_i^\eps(x),0,\cdots,0,U_i^\eps(x),0,\cdots,0)
	   \text{ which verifies }
	   \|R_i^\eps(\cdot)-r_i^*\|_\infty+\|U_i^\eps(\cdot)-u_i^*\|_\infty \leq C\eps.$$
	  Moreover, $W_i^\eps$ is hyperbolic and is stable 
		    if $r_i^*<r_j^*$ for all $j\in\{0,\cdots,N \}\setminus \{i\}$ 
		    and unstable else.
\item[(iii)]Suppose that $r_i^*\neq r_j^*$ for all $i\neq j$ and $r_i^*<r_0^*$ for all $i\geq 1$.  
	    Then the system $S_\eps$ has exactly $N+1$ non-negative stationary solutions:
	    $W_i^\eps(x),\; i=0,\cdots,N$.
	    Moreover, all these solutions are hyperbolic and exactly one of them is stable: 
	    $W_{i_0}^\eps$ where $r_{i_0}^*=\min\{r_0^*,\cdots,r_N^*\}.$
\end{itemize}
\end{theorem}
If in addition, the global dynamics of $S_0^c$ is known, then so  is the global dynamics of $S_\eps$.
The system $S_0^c$ being a homogeneous chemostat model, for a large choice of functions 
$\widetilde{f_i}$, it verifies the Competitive Exclusion Principle (CEP).\\
 More precisly,
 it is known that if $r_i^*>r_0^*$ then $u_i(t)\to 0$ as $t\to +\infty$, 
 therefore if $r_0^*<r_i^*$ for all $i\geq 1$, then the only steady state $(r_0^*,0,\cdots,0)$ of $S_0^c$ is a global attractor (in the nonnegative cadrant $\R_+^{N+1}$). 

\noindent If for some $i\geq 1$ one has $r_i^*<r_0^*$ then the global dynamics of $S_0^c$ is known under some additional assumptions.
Here, we make the following assumption on $S_0^c$ which is sufficient\footnote{
The proposition \ref{hyp0} holds true under more general hypothesis, see the monograph of Smith and Waltmann  \cite{waltbook}. Indeed, a well known conjecture asserts that the CEP holds true under 
 the simpler hypothesis of monotonicity of the functions $f_i$. This result is proven for equal mortalities in Amstrong and McGehee \cite{ArMg} (1980). 
In the case of different mortalities, this result is proven using Lyapunov functionals when the functions $\widetilde{f_i}$ 
verify some additional assumption. 
We refers to Hsu \cite{Hsu} (1978), Wolkowicz and Lu \cite{WL} (1992), Wolkowicz and Xia \cite{WX} (1997) and Li \cite{LI} (1998) for 
historical advances on this topic. See also  Sari and Mazenc \cite{SaMa} (2011) for recent results on this subject.
  }
 to ensure that $S_0^c$ satisfies the CEP.
\begin{assumption}\label{hyp0true}
One assumes that $\widetilde{f_i}$ is increasing and that either
\begin{itemize}
\item[(i)] For each $i\in\{1,\cdots,N\}$ one has $\widetilde{m_i}=\widetilde{m_0}>0$.
\item[(ii)] For each $i\in\{1,\cdots,N\}$, $\widetilde{f_i}$ reads  $\widetilde{f_i}(r)=c_i f(r)$ for some (increasing) function $f$ and positive constant $c_i$.
\item[(iii)] For each $i\in\{1,\cdots,N\}$, $\widetilde{f_i}$ reads $\widetilde{f_i}(r)=\frac{c_ir}{k_i+r}$ for some 
positive constants $c_i$ and $k_i$.
\end{itemize}
\end{assumption}

Under this assumption,  the asymptotic dynamics of $S_0^c$ (and all its sub-systems) are known in the following sense 
(see \cite{waltbook} for a proof).
\begin{proposition}[CEP for the aggregated system $S_0^c$ (see \cite{waltbook})]\label{hyp0}
Assume that the assumption   \eqref{hyp0true} holds true. 
Let $(r(t),u_1(t),\cdots,u_N(t))$ be a solution of $S_0^c$ 
with nonnegative initial conditions.\\
Define the set $J=\{0\}\cup\{j\in\{1,\cdots,N\},\; u_j(0)>0,\; r_j^*< r_0^*\}$
and the number  ${\displaystyle\widehat{r}=\min_{j\in J}(r_j^*)}$. We have\\
\begin{enumerate}[(i)]
\item ${\displaystyle\lim_{t\to+\infty}r(t)= \widehat{r}}$ and $\forall i\notin J,\; {\displaystyle \lim_{t\to +\infty} u_i(t)=0.}$ 
 \item In particular, if $J=\{0\}$ then $p_0^*:=(r_0^*,0,\cdots,0)$ is a global attractor in $\R_+^{N+1}$.
 \item If for some $j_1\in J\setminus \{0\}$ one has $r_{j_1}^*<r_{j}^*$ for any $j\in J\setminus \{j_1\}$ then
$$\lim_{t\to+\infty}u_{j_1}(t)=\frac{\widetilde{m_0}}{\widetilde{m_{j_1}}}\left(r_0^*-r_{j_1}^*\right)\text{ and } \lim_{t\to+\infty}u_{j}(t)=0,
\; \forall j\in J\setminus\{0,j_1\}$$
\end{enumerate}

\end{proposition}
Note that, from the assumption \ref{chap5hypf}, $\widetilde{f_i}$ is increasing. In practice, one has to  compute the functions $\widetilde{f_i}$ 
explicitly to verify the assumption \ref{hyp0true}.
Here are  some explicit examples ensuring that the assumption \ref{hyp0true}  holds true.
\begin{itemize}
\item[(i)] Assume that $m_i(x)=m_0(x)$ for any $x\in \Omega$. Then the case $(i)$ of the assumption \ref{hyp0true} occurs.
\item[(ii)] Assume that $f_i(x,R)=C_i(x)f(R)$ for some smooth positive functions $C_i:\Omega\to\R^+$ and $f:\R_+\to \R_+$. Then 
$$\widetilde{f_i}(r)=c_if(r),\;\text{ where } c_i=\frac{1}{|\Omega|}\int_\Omega C_i(x)dx$$ and the case $(ii)$ of the assumption \ref{hyp0true} occurs.
\item[(ii')] Assume that for each $i\geq 2$,  $f_i(x,R)=c_i f_1(x,R)$ for some positive constant $c_i$. Then 
$\widetilde{f_i}(r)=c_i\widetilde{f_1}(r)$ and the case $(ii)$ of the assumption \ref{hyp0true} occurs.
\item[(iii)] Assume that $f_i(R,x)=\frac{C_i(x)R}{k_i+R}$ where $k_i$ is a positive constant. Then 
$$\widetilde{f_i}(r)=\frac{c_ir}{k_i+r}\;\text{ where } c_i=\frac{1}{|\Omega|}\int_\Omega C_i(x)dx$$
and the cases $(iii)$ of the assumption \ref{hyp0true} occurs.
\end{itemize}

\noindent Now, we are  in position to state our main result. Let us denote the non-negative cadrant of $\left(C^0(\overline{\Omega})\right)^{N+1}$ by
$$Q=\left\{V(\cdot)\in C^0(\overline{\Omega}),\;V(x)\geq 0,\,\forall x \in \overline{\Omega}\right\}^{N+1}.$$
Thanks to the crucial uniform boudedness result (theorem \ref{chap5thmaj}), one obtains the global dynamics in $Q$ for small $\eps$.
\begin{theorem}[CEP for the original system $S_\eps$]\label{mains2}
Assume that the assumptions \eqref{chap5hyp0} and \eqref{chap5hypf} hold true. For each $i$, 
denote $\bW_i^\eps(x)$ the stationary solution of $S_\eps$ as defined in the Theorem \ref{mains1}.
There exists $\eps_0>0$ such that for all $\eps\in (0,\eps_0)$ and  initial data 
$\bW^\eps(\cdot,0)\in Q$,
one has the following properties.
\begin{itemize}
\item[(i)] Let $i\in\{1,\cdots,N\}$. If $r_i^*>r_0^*$ then ${\displaystyle\lim_{t\to \infty} \|U_i^\eps(\cdot,t)\|_\infty=0}$.
\item[(ii)] Assume that $r_0^*<r_i^*$ for all $i\geq 1$. Then every solution $\bW^\eps(x,t)$ of $S_\eps$  verifies
$$\lim_{t\to +\infty} \|\bW^\eps(\cdot,t)-\bW_0^\eps(\cdot)\|_\infty=0.$$ 
\item[(iii)] Assume that $r_1^*<r_i^*$  for all $i\neq1$ and that the assumption \ref{hyp0}  holds. 
Then every solution $\bW^\eps(x,t)$ of $S_\eps$ with  nonnegative initial data verifying $U_1^\eps(x,0)> 0$ for some $x\in \Omega$ verifies
$$\lim_{t\to +\infty} \|\bW^\eps(\cdot,t)-\bW_1^\eps(\cdot)\|_\infty=0.$$
\end{itemize}
\end{theorem}

\section{General results for slow-fast system}
\noindent In this section we state precisly the  Central manifold Theorem \ref{centralmanifold} and 
the Theorem of convergence towards the central manifold \ref{errorbound}. These theorems may be proved following \cite{Castella2009}. 
Next, we state and prove two general results for fast-slow systems: propositions \ref{main1general} and \ref{main2general}. 
These propositions are used in section \ref{proofs}
to prove the Theorems \ref{mains1} and \ref{mains2}.
\subsection{Central Manifold Theorem}\label{seccentralmanifold}

Let us begin by a version of the central manifold Theorem used in this paper. 
This Theorem claims the existence of an invariant manifold for the slow-fast system which allows to defined several reduced systems.

\begin{theorem}[Central manifold Theorem]\label{centralmanifold}
Let  $E$ and $F$ be two Banach spaces.
Define $\mathcal{F}_0(X,Y)\in C^1(E\times F;E)$ and $\mathcal{G}_0(X,Y)\in C^1(E\times F;F)$. 
One assumes that $\mathcal{F}_0$ and $\mathcal{G}_1$ are uniformly bounded as well than there first derivatives.
 Let $K$ be an operator with  domain  $\mathcal{D}(K)\subset F$. 
 One assumes that  $K$ generates an analytical semi-group   $exp(tK)$ of 
 linearly  operators on $F$ and that there exists  $\mu>0$ such that
$$\forall t\geq0, \quad\forall \eps\in(0,1],\quad 
\left\|exp\left(\frac{t}{\eps}K\right)Y\right\|_{F}\leq C\|Y\|_{F} exp\left(-\mu\frac{t}{\eps}\right).$$

\noindent For all initial condition $(x_0,y_0)\in E\times F$ and, for all $\eps\in(0,1]$, 
on defines  $X^\eps(t,x_0,y_0)\equiv X^\eps(t)$ and $Y^\eps(t,x_0,y_0)\equiv Y^\eps(t)$ the solution,  
for $t\geq 0$, of the differential system
\begin{equation*}
\begin{array}{cc}
S_\eps^{sf}&\left\{\begin{array}{l} 
\frac{d}{dt}X^\eps(t)=\mathcal{F}_0(X^\eps(t),Y^\eps(t)),\\
\frac{d}{dt}Y^\eps(t)=\mathcal{G}_1(X^\eps(t),Y^\eps(t))+\frac{1}{\eps}K Y^\eps(t) \\
 X^\eps(0)=x_0,\quad Y^\eps(0)=y_0.\\
\end{array}\right.\end{array}\end{equation*}

Then, there exists $\eps_0>0$ such that, for all $\eps\in(0,\eps_0)$, the system  $S_\eps^{sf}$ 
admit a central manifold  $\mathcal{M}^\eps$ in the following sense.\\

\noindent There exists a function $h(X,\eps)\in C^1(E\times [0,\eps_0];F)$ such that, for all $\eps\in]0,\eps_0]$, the set 
$\mathcal{M}^\eps=\{(X,h(X,\eps));X\in E\}$ is invariant under the semi flow generated by $S_\eps^{sf}$ for $t\geq 0$. 
Moreover, $$\|h(\cdot,\eps)\|_{L^\infty(E,F)}=O(\eps)\text{ as }\eps\to 0.$$
\end{theorem}

\noindent This Theorem provides the existence of a manifold  $\mathcal{M}^\eps$ which is invariant for the system  $S_\eps^{sf}$ 
and parametrized by the slow variable $X^\eps\in E$.  
In our application,  $E$ is finite dimensional so that the system {\it on} $\mathcal{M}_\eps$ is a finite dimensional system.  
After showing that the solutions are close to the central manifold, up to an exponentially small error term,
 we can reduce the study to a system {\it on} the invariant manifold $\mathcal{M}^\eps$. 
 This finite dimensional system  approach, in a sense that we specify below, the original problem. 
 
\noindent More precisly, let us define the following reduced system. We do not precise the initial data at this step.

$$\left(S_\eps^{[\infty]}\right) \quad \frac{d}{dt} X^{\eps,[\infty]}(t)=\mathcal{F}_0(X^{\eps,[\infty]}(t),h(X^{\eps,[\infty]}(t),\eps)),
\quad Y^{\eps,[\infty]}(t)=h(X^{\eps,[\infty]}(t),\eps)$$

When the original data lies on this manifold, $S_\eps^{[\infty]}$ describes  the exact dynamics of $S_\eps^{sf}$. 
In general $Y^\eps(0)\neq h(X^\eps(0),\eps)$ and the real solutions do not belong to $\mathcal{M}^\eps$. 
However, the next theorem state that, up to slightly modify the initial datum, the solution of $S_\eps^{sf}$ 
are exponentially close to the solution of $S_\eps^{[\infty]}$.\\

The exact calculation of the cental manifold is usually out of reach. 
A practical idea is to make approximate calculations. 
Theorem \ref{centralmanifold} ensures that $ h (X, \eps) = O (\eps) $. 
So, as a first approximation\footnote{
Indeed, $ h (X, \eps) $ admits an asymptotic expansion of the form $ h (X, \eps) = \sum_{k=1}^r \eps^kh_k (X) + O (\eps^{r+1}) $ 
which is explicitly calculable provided the functions $ \mathcal{F}_0 $ and $ \mathcal{G}_0 $ have $C^{r+1}$ smoothness.
The approximate $ h (X, \eps) \approx \sum_{k=1}^r \eps^k h_k (X) $ leads to the writing of reduced systems of order $ r $ (see \cite{Castella2009}). This paper focus only on  the case $ r =  0$.
}, $ h(X,\eps) \approx 0$ and we obtain the following reduced system

$$\left(S_\eps^{[0]}\right) \quad \frac{d}{dt} X^{\eps,[0]}(t)=\mathcal{F}_0(X^{\eps,[0]}(t),0),\quad Y^{\eps,[0]}(t)=h(X^{\eps,[0]}(t),\eps).$$

\noindent In addition to the exponentially small error term between the solutions of $S_\eps^{sf}$ and the central manifold $\mathcal{M}^\eps$, 
the following Theorem describes the error (more precisly a shadowing principle) between the reduced systems  
$S_\eps^{[\infty]}$ and  $S_\eps^{[0]}$ and the original system  $S_\eps^{sf}$.

\begin{theorem}[error bounds between the reduced systems  and the original system]\label{errorbound}
Under the assumptions and the notations of the Theorem \ref{centralmanifold}, for any exponant  $0<\mu'<\mu$ and any initial data $(X_0,Y_0)\in E\times F$, the following assertions hold true.
\begin{enumerate}[(i)]
\item {\bf Exponential convergence towards the central manifold.}\\
There exists a constant $C>0$ such that 
$$\forall t\geq 0,\quad \|Y^\eps(t)-h(X^\eps(t),\eps)\|_F \leq C exp\left(-\mu' \frac{t}{\eps}\right).$$

\item {\bf Shadowing principle for $\left(S_\eps^{[\infty]}\right)$}.\\
For any $T>0$, there exist an initial data  $X_0^\eps$, depending on $T$ and $\eps$-close to $X_0$ and a constant $C_T>0$, such that the solution of the reduced system $S_\eps^{[\infty]}$, with initial data $X^{\eps,[\infty]}(0)=X_0^\eps$ and $Y^{\eps,[\infty]}(0)=h(X_0^\eps,\eps)$, satisfies the following error estimate 
$$\forall t\in[0,T],\quad \|X^\eps(t)-X^{\eps,[\infty]}(t)\|_E+\|Y^\eps(t)-Y^{\eps,[\infty]}(t)\|_F\leq  C_Texp\left(-\mu'\frac{t}{\eps}\right),$$ 
where $C_T>0$ is independent of $t\geq 0$ and $\eps$. If moreover there exists  $M>0$ independent of $t$ and $\eps$ such that, for all $t>0$, $\|X^\eps(t)\|_E\leq M$,  then we can take $T=+\infty$.

\item {\bf Shadowing principle for $\left(S_\eps^{[0]}\right)$}.\\
For any $T>0$, there exist an initial data  $X_0^\eps$, depending on $T$ and $\eps$-close to $X_0$ and a constant $C_T>0$, such that the solution of the reduced system $S_\eps^{[0]}$, with $X^{\eps,[0]}(0)=X_0^\eps$, satisfies the following error
estimate 
$$\forall t\in[0,T],\quad \|X^\eps(t)-X^{\eps,[0]}(t)\|_E+\|Y^\eps(t)-Y^{\eps,[0]}(t)\|_F\leq C_T\left(\eps+ exp\left(-\mu'\frac{t}{\eps}\right)\right),$$
where $C_T>0$ is independent of $t\geq 0$ and $\eps$. If moreover there exists  $M>0$ independent of $t$ and $\eps$ such that, for all $t>0$, $\|X^\eps(t)\|_E\leq M$,  then we can take $T=+\infty$.
\end{enumerate}
\end{theorem}

\noindent This Theorem means that, up to slightly modify the initial datum, the original system is well 
described by the reduced systems when $\eps$ is small enough. This allows us to study the qualitative behavior of 
solutions of the original system by working on finite dimensional systems. 
\begin{remark}\label{remarkx0}
The initial data $X_0^\eps$ is constructed as follows.\\
First for a fixed $T>0$, one chooses $X_0^\eps(T)=X_T^{\eps,[0]}(0)$ as the only initial conditions such that
the solution of $\frac{d}{dt}X_T^{\eps,[0]}(t)=\mathcal{F}_0(X_T^{\eps,[0]}(t),0)$ verifies $X_T^{\eps,[0]}(T)=X^\eps(T)$.\\
Now if  $X^\eps$ is uniformly bounded in $E$, independently of $t$ and $\eps$, 
then $X_T^{\eps,[0]}$ and $\frac{d}{dt} X_T^{\eps,[0]}$ are bounded as well. By the Ascoli Theorem, 
one can choose a sequence of trajectories $X_T^{\eps,[0]}$ which converges as $T\to+\infty$.
This allows us to define $X_0^\eps=lim_{T\to\infty} X_T^{\eps,[\infty]}(0)$.\\
As a consequence, if $S_\eps^{sf}$ conserves the line $X_i=0$, then for any initial  data satisfying $X_i^\eps(0)\geq 0$ 
one see that  $X_i^{\eps,[0]}(T):=X_i^\eps(T)\geq 0$ for any fixed $T>0$. If in addition, $S_\eps^{[0]}$ conserves the line $X_i=0$,
 this implies that the $i$th componant $X_{0,i}^\eps(T):=X_{T,i}^{\eps,[0]}(0)$ is nonnegative.
This fact remains obviously true by passing to the limit $T\to +\infty$. In conclusion, if $X_i^\eps(0)\geq 0$ then one has $X_{0,i}^\eps\geq 0$. 
This fact is essential in order to deal with global dynamics in the positive cone. 
\end{remark}
\subsection{General consequences}\label{generals}
The aim of this section is to prove the two below stated general results on slow-fast system: propositions \ref{main1general} and \ref{main2general}. 
These  propositions  are the key in the proofs of our main results, theorems \ref{mains1} and \ref{mains2}.
In order to prove these two propositions, we start by the three following lemmas.
%

\noindent The  first lemma uses the invariance of the central manifold and is already noted in \cite{Ca}.
\begin{lemma}\label{Csteadystate}
Each stationary solution of $S_\eps^{sf}$ lies on $\mathcal{M}^\eps$.
\end{lemma}
\begin{proof}
Let $P^{\eps}=(X^{\eps},Y^{\eps})\in E\times F$ be a  stationary solution of $S_\eps^{sf}$. 
The invariance of the central manifold implies that $(X^{\eps},h(X^{\eps},\eps))$ is a stationary solution of $S_\eps^{sf}$. 
By the theorem \ref{errorbound}, it comes 
$$\|Y^{\eps}-h(X^{\eps},\eps)\|_F\leq Cexp(-\mu \frac{t}{\eps})$$
and so, by passing to the limit $t\to+\infty$,  
$$Y^{\eps}=h(X^{\eps},\eps).$$
\end{proof}

\noindent Hence, the complete description of the stationary solutions of the {\it finite dimensional system}
$$S_\eps^c\;:\; \frac{d}{dt} X^{\eps,[\infty]}(t)=\mathcal{F}_0(X^{\eps,[\infty]}(t),h(X^{\eps,[\infty]}(t),\eps))$$
provides a  complet  description of
the stationary solutions of the slow-fast system $S_\eps^{sf}$. 

\paragraph{}
Despite the fact that the system  $S_\eps^{c}$ is finite dimensional, 
it is not explicit and difficult to study directly. 
But it can generically be seen as a regular perturbation of  $S_0^{c}$ 
and stationary solutions can then be easily reconstructed by local inversion.
\begin{lemma}\label{stationaryS0}
Assume that $p^0$ is a stationary asymptotically linearly stable (unstable) solution of $S_0^{c}$. 
Then there exists $\eps_1>0$ such that for all $\eps\in[0,\eps_1]$, there exists a  stationary point $p^\eps\in E$ of $S_\eps^{c}$ 
which is asymptotically linearly stable (resp. unstable) and $\eps \mapsto p^\eps$ is a $C^1$ function from $[0,\eps_1]$ to $E$.
Moreover, $p^\eps$ is the only stationary solutions of $S_\eps^{c}$ in a neighborhood of $p^0$.\\

\end{lemma}
\begin{proof}
A simple application of the implicit function theorem on the function $(X,\eps)\mapsto \mathcal{F}_0(X,h(X,\eps))$ 
shows both the existence of the $C^1$ map $\eps\mapsto p^\eps$ and the uniqueness. The systems $S_\eps^{c}$ being finite dimensional, 
a simple perturbation argument shows that $p^\eps$ is asymptotically linearly stable (unstable).
\end{proof}\\

Thanks to the regularity of $\mathcal{F}_0$, linear asymptotic stability implies asymptotic stability. 
Hence, if $p^0$ is linearly asymptotically stable, then $p^\eps$ is asymptotically stable. 
In fact, a stronger result holds : the size of the basin of attraction can be chosen independently on $\eps$.
 This is used strongly in the sequel to deduce both local and global stability properties of the stationary solutions of $S_\eps$ 
 from the corresponding results for $S_0^{c}$.\\
\begin{lemma}\label{bassin}

Define $S^\eps(t)$ the one-parameter group associated to $S_\eps^{c}$. That is
$$S^\eps(t)X_0=X^\eps(t)$$
where $X^\eps(t)$ is the only solution of $S_\eps^{c}$ with initial data $X_0$.\\
If $p^0$ is linearly asymptotically stable for $S_0^c$, then
\begin{equation*}
\exists \eps_0>0,\;\exists r>0,\; \forall \eps\in[0,\eps_0],\; \forall w_0\in B(p^\eps,r),\; \lim_{t\to+\infty}\|S^\eps(t)w_0-p^\eps\|=0
\end{equation*}
\end{lemma}

\begin{proof}

Since the linear stability implies the (local) stability, the lemma \ref{stationaryS0} yields that for all $\eps\in(0,\eps_1)$ there exists $r>0$ such that 
$$X_0\in B(p^\eps,r)\Rightarrow \lim_{t\to +\infty}\|S(t)X_0-p^\eps\|_E=0.$$
So one can define 
\begin{equation}\label{reepsdef}
r_\eps=sup\{r>0, \forall w\in B(p^\eps,r),\;\text{s.t. } \lim_{t\to+\infty} \|S^\eps(t)w-p^\eps\|_E=0\}.
\end{equation}
The lemma holds true if $\liminf_{\eps\to 0} r_\eps >0$. Let us argue by contradiction.\\

\noindent Suppose that $\liminf_{\eps\to 0} r_\eps =0$, then there exists three 
sequences 
$\eps_n\to 0$, $r_{\eps_n}\to 0$ and $w_n\in E$ 
verifying 
\begin{equation}\label{estih}
r_{\eps_n}\leq \|w_n-p^{\eps_n}\|_E\leq 2r_{\eps_n},
\end{equation}
such that
\begin{equation}\label{contradiction}
\limsup_{t\to+\infty} \|S^{\eps_n}(t)w_n-p^{\eps_n}\|_E>0.
\end{equation}

\noindent We claim that 

\begin{equation}\label{claim1}
\forall t\geq 0,\quad \|S^{\eps_n}(t)w_n-p^{\eps_n}\|_E\geq r_{\eps_n}.
\end{equation}

\noindent Indeed, arguing by contradiction, assume that there exists $t_0\geq 0$ such that

$$\|S^{\eps_n}(t_0)w_n-p^{\eps_n}\|_E<r_{\eps_n}$$

\noindent Therefore one gets for each $t>t_0$,
$$\|S^{\eps_n}(t)w_n-p^{\eps_n}\|_E=\|S^{\eps_n}(t-t_0)S^{\eps_n}(t_0)w_n-p^{\eps_n}\|_E$$

\noindent So that, by \eqref{reepsdef},
$$\lim_{t\to +\infty}\|S^{\eps_n}(t)w_n-p^{\eps_n}\|_E=0,$$
which contradicts \eqref{contradiction}. It follows that \eqref{claim1} holds.

\noindent Now, denote $h_n=w_n-p^{\eps_n}$ and remark that $S^{\eps}(t) p^\eps=p^\eps$. One gets for all $t\geq 0$,
\begin{equation}\label{eq313}
r_{\eps_n}\leq \|S^{\eps_n}(t)h_n\|_E\leq \|S^{\eps_n}(t){h_n}-S^{0}(t)h_n\|_E+\|S^0(t) h_n\|_E
\end{equation}

Take any $T>0$,  the Gronwall Lemma together with global Lipschitz property of $\mathcal{F}_0$ and $h$ yields for all $t\in[0,T]$
\begin{equation}\label{eq1}
\|S^\eps(t)X_0- S^0(t)X_0\|_E\leq \eps C_T \|X_0\|_E
\end{equation}
for some positive constant $C_T$ independent  on $t$ and $X_0$. Therefore
$$\|S^{\eps_n}(t)h_n\|_E\leq \eps_nC_T \|h_n\|_E+\|S^0(t) h_n\|_E.$$ 
Divide \eqref{eq313} by $\|h_n\|_E$, using \eqref{estih} and passing, up to a subsequence, to the limit $n\to +\infty$, one obtains
\begin{equation}\label{contradict}
\forall t\in(0,T),\quad\frac{1}{2}\leq \lim_{n\to+\infty}\frac{1}{\|h_n\|_E}\|S^0(t) h_n\|_E \leq\||e^{t A}\||\end{equation}
where $A=D_X\mathcal{F}_0(p^0,0)$.

The asymptotic linear stability of $p^0$, reads $\sigma(A)\subset\{\lambda\in \C,\; \Re(\lambda)\in ]-\infty,-\beta]\}$ 
for some $\beta>0$ so that
$$\lim_{t\to +\infty}\||e^{t A}\||=0$$ 
which yields to a contradiction by taking $T$ and $t$ big enough in \eqref{contradict}.
\end{proof}

\noindent One can now state the first  proposition describing completly the stationary solutions of $S_\eps^{sf}$.
\begin{proposition}\label{main1general}Under the assumptions of theorem \ref{centralmanifold}, there exists $\eps_0>0$ 
such that for each $\eps\in(0,\eps_0)$ the following holds true.
\begin{itemize}

\item[(i)] Assume that  $S_0^{c}$ has one stationary solution $p^0$ which is  
hyperbolic.%
Then $S_\eps^{sf}$ has one stationary solution $P^\eps:=(p^\eps,h(p^{\eps},\eps))$ which is hyperbolic and verifies 
${\displaystyle \lim_{\eps\to 0} \|p^{\eps}-p^0\|_E=0.}$\\
 $P^\eps$ is called the stationary solution corresponding to $p^{0}$.
 
 \item[(ii)] Assume that $S_0^{c}$ has one linearly asymptotically  stable 
 (resp. unstable)
solution. Then the corresponding stationary solution of $S_\eps^{sf}$ is linearly asymptotically stable (resp. unstable).

\item[(iii)] If all the stationary solution of $S_0^{c}$ are hyperbolic,
 then $S_0^{c}$ has a finite number $m$ of stationary solution and
  $S_\eps^{sf}$ has exactly $m$ stationary solutions.
  
 \end{itemize}

\end{proposition}

\

\noindent \begin{proof}
Proof of  $(i)$. By the lemma \ref{stationaryS0}, one knows that  there exists  $p^\eps$  an hyperbolic stationary solution of $S_\eps^{c}$. 
It follows that  $P^\eps:=(p^\eps,h(p^\eps,\eps))$ is a stationary solution of 
$S_\eps^{sf}$.\\
Proof of  $(ii)$. Assume that $p^0$ is a linearly asymptotically stable (resp. unstable) stationary solution of $S_0^c$. 
By the lemma \ref{stationaryS0}, $p^\eps$ is a linearly asymptotically stable (resp. unstable) stationary solution of $S_\eps^c$. 
It remains to proof that if $p^\eps$ is stable (resp. unstable) for $S_\eps^{c}$ 
then so  is $P^\eps$ for $S_\eps^{sf}$. 
If $p^{\eps}$ is an unstable stationary solution of $S_\eps^{c}$, then $P^\eps$ is obviously an unstable stationary solution of $S_\eps^{sf}$.

\noindent Let us show that, if $p^{\eps}$ is a stable stationary solution of $S_\eps^{c}$,
 then $P^\eps$  is a stable stationary solution of $S_\eps^{sf}$. This is the main difficulties of this proof. 
 We solve this problem\footnote{Indeed, this is a general fact for central manifold  as point out by Carr \cite{Ca}.} by using  lemma \ref{bassin}. \\\
Denote   $Z^\eps(t)=(X^\eps(t),Y^\eps(t))$ the only solution of $S_\eps^{sf}$ with initial data 
$Z^\eps(0)=(X_0,Y_0)$ in a neighborhood (remaining to determine) of $P^{\eps}$ in $E\times F$ 
and $Z^{\eps,[\infty]}(t)=(X^{\eps,[\infty]}(t),h(^{\eps,[\infty]}(t),\eps))$ the only solution of $S_\eps^{[\infty]}$ 
with initial data $Z^{\eps,[\infty]}(0)=(X_0^\eps,h(X_0^\eps,\eps))$ 
given in the Theorem \ref{errorbound}- (iii). Recall that $\|X_0^\eps-X_0\|_E=O(\eps)$.
One gets 
\begin{equation*}
\begin{split}
\|Z^{\eps}(t)-P^{\eps}\|_{E\times F}&:=\|X^\eps(t)-p^\eps\|_E+\|Y^\eps(t)-h(p^\eps,\eps)\|_F\\
							&\leq \|Z^{\eps}(t)-Z^{\eps,[\infty]}(t)\|_{E\times F}
							+\|Z^{\eps,[\infty]}(t)-P^{\eps}\|_{E\times F}.
\end{split}\end{equation*}

\noindent Let $r>0$ be the size of the  basin of attraction define in the lemma \ref{bassin}. 
$r$ is independent of $\eps$. 
If $\|Z^{\eps}(0)-P^{\eps}\|_{E\times F}\leq r/3$,  then one gets
$$\|X_0^\eps-p^{\eps}\|_{E\times F}=\|X_0^\eps-X_0\|_E+\|X_0-p^\eps\|_F\leq r/2$$
for small enough $\eps$. \\
Therefore,  Lemma  \ref{bassin} yields

$$\lim_{t\to+\infty}\|X^{\eps,[\infty]}(t)-p^{\eps}\|_E=0$$
and then by continuity of $h$,
$$\lim_{t\to+\infty}\|Z^{\eps,[\infty]}(t)-P^{\eps}\|_{E\times F}=0.$$
Finally, by the Theorem \ref{errorbound}, for some positive constants  $C$ and $\mu'$, one gets
$$\|Z^{\eps}(t)-Z^{\eps,[\infty]}(t)\|_{E\times F}\leq Cexp(-\mu'\frac{t}{\eps})\to 0\text{ as } t\to +\infty,$$
which shows that
$$\|Z^{\eps}(t)-P^{\eps}\|_{E\times F}\to 0\text{ as } t\to+\infty,$$
and end the proof of the stability of $P^{\eps}$ for $S_\eps^{sf}$.
\end{proof}
\paragraph{}

The last result of this section describes the {\it asymptotic dynamics} of $S_\eps^{sf}$ when the global dynamics of $S_0^c$ is known. 

\begin{proposition}\label{main2general}
 Suppose that the assumption of the theorem \ref{centralmanifold} are verified. Set
 $\eps_0>0$ the (small) scalar such that for all $\eps\in (0,\eps_0)$ the conclusion of theorems \ref{centralmanifold} and \ref{errorbound} occur.\\
 Let $\eps\in(0,\eps_0)$ and for any initial condition $Z_0=(X_0,Y_0)\in E\times F$, define $X_0^{\eps}(Z_0)\in E$ be a modified initial data 
 appearing in  the theorem \eqref{errorbound}-(iii).\\
 Assume that there exists three set $\mathcal{Q}\in E$, $Q_E\in E$ and $Q_F\in F$  satisfying  the three following assumptions.
 \begin{itemize}
   
  \item[(i)]$S_0^c$ admits one hyperbolic stationary solution $p^0\in \mathcal{Q}$ 
  which is a global attractor in $\mathcal{Q}$ 
  for the dynamic of $S_0^c$.
  \end{itemize}
 Let $P^\eps:=(p^\eps,h(p^\eps,\eps))\in E\times F$ be the corresponding stationnary solution for $S_\eps^{sf}$.
\begin{itemize}
 \item[(ii)]   
 For any initial condition $Z_0=(X_0,Y_0)\in Q_E\times Q_F$, the modified initial data $X_0^{\eps}(Z_0)$ belongs to $\mathcal{Q}$.
 \end{itemize}
Then, for any initial condition $Z_0\in Q_E\times Q_F$, one have $\|\bW^\eps(\cdot,t)-P^\eps(\cdot)\|_{E\times F} \to 0$ as $t\to +\infty$.
 \end{proposition}

\begin{remark}Since the modified initial data $X_0^\eps$ is $\eps$-close to $X_0$, 
if $Q_E \subset \text{int}(\mathcal{Q})$, then for $\eps$ small enough, the assumption $(ii)$ is satisfied. 
The only difficulty in the application is when $Q_E\cap \partial \mathcal{Q}\neq \emptyset$ 
which may  occur when we deal with dynamics in the nonnegative cadrant, see  lemma \ref{lemmarepuls}. 
\end{remark}


\begin{proof}
Let $p^0$ be an linearly asymptotically stable stationary solution of $S_0^c$. 
By the Theorem \ref{main1general}, the  steady state $P^{\eps}=(p^{\eps},h(p^{\eps},\eps))$ exists  and is a  local attractor.
Besides,   by the lemma \ref{stationaryS0} and the smoothness of $h$, one gets for some positive constant  $C'$ independent on  $\eps$,
\begin{equation}\label{errPet} 
\|p^{\eps}-p^0\|_E+\|h(p^{\eps},\eps)-h(p^{0},\eps)\|_F\leq C'\eps.
\end{equation} 
Let  $Z^\eps(t)=(X^\eps(t),Y^\eps(t))$ be the solution of $S_\eps^{sf}$ with initial data $(X_0,Y_0)\in Q_E\times Q_F$ 
 and $Z^{\eps,[0]}(t)=(X^0(t),h(X^0(t),\eps))$ be the solution of $S_\eps^{[0]}$  
with initial data $Z^{\eps,[0]}(0):=(X_0^\eps,h(X_0^\eps,\eps))$ given in the Theorem \ref{errorbound}. \\
By  Theorem  \ref{errorbound}, it comes for some positive constants  $C$ and $\mu'$ and any $t\geq 0$ and small enough $\eps$, the bound
\begin{equation}\label{errorCeps}
\|Z^\eps(t)-Z^{\eps,[0]}\|_{E\times F}\leq C(\eps+ exp(-\mu'\frac{t}{\eps})).
\end{equation}
Let  $r>0$ be the size of the basin of attraction given in the lemma \ref{bassin}. 
By the assumption $(ii)$,  $X_0^\eps\in\mathcal{Q}$ and by the assumption $(i)$, $p^0$ is a global attractor of $S_0^{c}$ in $\mathcal{Q}$. This implies 
\begin{equation*}
\exists T>0, \forall t\geq T, \|X^{0}(t)-p^0\|_E\leq r/4.
\end{equation*}
By the continuity of $X\mapsto h(X,\eps)$, this yields 
\begin{equation}\label{globattract0}
\exists T>0, \forall t\geq T, \|Z^{\eps,[0]}(t)-P^{0}\|_{E\times F}
:=\|X^{0}(t)-p^0\|_E+\|h(X^{0}(t),\eps)-h(p^0,\eps)\|_F\leq r/3.
\end{equation}
Besides, for all  $t\geq 0$,
\begin{equation*}
\|Z^\eps(t)-P^{\eps}\|_{E\times F}\leq \|Z^\eps(t)-Z^{\eps,[0]}(t)\|_{E\times F}+\|Z^{\eps,[0]}(t)-P^{0}\|_{E\times F}+\|P^0-P^{\eps}\|_{E\times F}
\end{equation*}
The inequalities \ref{errPet}, \ref{errorCeps} and \ref{globattract0}, imply that there exists 
a constant $C''$ independent of $\eps$ and of $r$ such that,
\begin{equation}\label{eqglobfinal}
\exists T>0, \forall t\geq T, \|Z^\eps(t)-P^{\eps}\|_{E\times F}\leq C''(\eps+exp(-\mu\frac{t}{\eps}))+r/3.
\end{equation}
Choosing  $\eps$ small enough such that $C''(\eps+exp(-\mu\frac{T}{\eps}))\leq r/6$, \eqref{eqglobfinal} yields 
\begin{equation}\label{prochebassin}
\exists T>0, \forall t\geq T, \|Z^\eps(t)-P^{\eps}\|_{E\times F}\leq r/2.
\end{equation}
Arguing as in the proof of the Theorem \ref{main1general}, if $\eps$ is small enough, 
\eqref{prochebassin} implies $\|Z^\eps(t)-P^{\eps}\|_{E\times F}\to 0$ 
as needed.
\end{proof}

\section{Proofs of Theorems \ref{mains1} and \ref{mains2}}\label{proofs}
In this section, we begin by showing that the  Theorems \ref{centralmanifold} and \ref{errorbound} 
apply to our particular system $S_\eps$. Then we give the proof of the main results
\subsection{Application of the Central Manifold theorem}
The precise definitions of the operators $A_i^\infty$ and $K^\infty=diag(A_i^\infty)$ are given in  section \ref{sec22} 
as well as the definitions of  the banach spaces
$E=ker(K^\infty)=\R^{N+1}$ and $F=Im(K^\infty)$.
In the case of the system $S_\eps$, one gets explicitly, with the notation of the section \ref{sec22}, 
$$X^\eps:=(r^\eps,u_1^\eps\ldots,u_N^\eps)=\Pi_E\left( R^\eps, U_1^\eps,\cdots, U_N^\eps\right)
:=\left(\frac{1}{|\Omega|}\int_\Omega R^\eps,\frac{1}{|\Omega|}\int_\Omega U_1^\eps,\cdots, \frac{1}{|\Omega|}\int_\Omega U_N^\eps\right)$$

$$
Y^\eps(x):=(Y_0^\eps(x),\ldots,Y_N^\eps(x))=\Pi_F\left( R^\eps, U_1^\eps,\cdots, U_N^\eps\right)(x)
:= \left( R^\eps(x)-r^\eps, U_1^\eps(x)-u_1^\eps,\ldots, U_N^\eps(x)-u_N^\eps\right).\\
$$
Of course, with these notations, one has $X^\eps+Y^\eps=\left( R^\eps, U_1^\eps,\cdots, U_N^\eps\right)$. Finally, for any $x\in\Omega$,
$$\mathcal{F}(X^\eps+Y^\eps)(x)=
\left(\begin{array}{c}
I(x)-m_0(x)\big(r^\eps+Y_0^\eps(x)\big)-\sum\limits_{i=1}^N f_i(r^\eps+Y_0^\eps(x),x)\big(u_i^\eps+Y_i^\eps(x)\big) \\
 \Big(f_1(r^\eps+Y_0^\eps(x),x)-m_1(x)\Big)\big(u_1^\eps+Y_1^\eps(x)\big) \\
\vdots\\
 \Big(f_N(r^\eps+Y_0^\eps(x),x)-m_N(x)\Big)\big(u_N^\eps+Y_N^\eps(x)\big) \end{array}\right)
$$
and
$$\mathcal{F}_0(X^\eps,Y^\eps)=\Pi_E\mathcal{F}(X^\eps+Y^\eps) \text{ and } \mathcal{G}_1(X^\eps,Y^\eps)(x)=\Pi_F\mathcal{F}(X^\eps+Y^\eps)(x) .$$
Note that 
$$\mathcal{F}_0\;:\; E\times F\to E\;\text{ and }\; \mathcal{G}_1\;:\; E\times F\to F.$$


We first show that the operator $K=diag(A_i)$ define a $C^0$ semi-group of contraction on $F$.

\noindent

The assumed smoothness of $\partial \Omega$ implies that  the operator $A_i^\infty$ generates a 
$C^0$ semi-group of contraction on $C^0(\overline{\Omega})$ (see \cite{markus}). 
Denoting $exp(t A_i^\infty)$ this semi-group, this reads

$$\forall t\geq 0,\; \|exp\left(tA_i^\infty\right) v\|_\infty\leq \| v\|_\infty. $$

The following lemma is a well know result  using   the gap between the two first eigenvalues of $A_i^\infty$.
\begin{lemma}\label{lemmapazy}
 The restriction $A_i$ of $A_i^\infty$ to the subspace $\widetilde{F}:=\{u\in C^0(\overline{\Omega}),\,\int_\Omega u=0\}$ is the generator of a $C^0$ semi-group of strict contraction $exp(tA_i)$ on ${\widetilde{F}}$ verifying for some $\mu_i>0$
\begin{equation}\label{expdecay}
\forall v\in \widetilde{F},\;\|exp(t A_i) v\|_\infty\leq  e^{-\mu_i t}\|v\|_\infty.
\end{equation}
\end{lemma}

\begin{proof}
 $\widetilde{F}$ is closed in $C^0(\overline{\Omega})$ and is clearly invariant under $exp(tA_i^\infty)$. 
 It follows (Pazy \cite{Pazy} p. 123) that $A_i$ is the generator of a $C^0$ 
 semi-group of contraction on $\widetilde{F}$.\\
 It is well known that the spectrum $\sigma(-A_i^\infty)$ is a sequence of real nonnegative scalars
 $$0=\lambda_0<\lambda_1\leq \cdots $$
Since $\sigma(A_i)\subset \sigma(A_i^\infty)$ and $0\notin \sigma(A_i)$ 
one see that  $\sigma(A_i)\subset ]-\infty,-\lambda_1]$ 
and an application of the Theorem 4.3 p 118 in Pazy \cite{Pazy} end the proof.
\end{proof}

Noting $\mu=\min \{\mu_0,\cdots,\mu_N\}$ where $\mu_i$ is as in \eqref{expdecay} and  
$K:=diag(A_i)$, and $F=\widetilde{F}^{N+1}$, the  lemma \ref{lemmapazy} implies directly

\begin{proposition}\label{expdecayK}
$K$ is the generator of a $C^0$ semi-group $exp(tK)$ on $F$ verifying
$$\|exp(tK) v\|_F\leq e^{-\mu t} \|v\|_F.$$
\end{proposition}

\noindent Now, we show that the functions $\mathcal{F}_0=\Pi_E\mathcal{F}$ and $\mathcal{G}_1=\Pi_F \mathcal{F}$  are smooth enough. 

\begin{lemma}
The functions $\mathcal{F}_0$ and $\mathcal{G}_1$ have $C^1$ smoothness when acting on $E\times F$.
\end{lemma}
\begin{proof}
By assumption 2.1 and 2.2, $\mathcal{F}$ is $C^1$ from $E\oplus F$ into itself. 
The only difficulty is the presence of the linear operators $\Pi_E$ and $\Pi_F$. 
Since $\mathcal{G}_1=\mathcal{F}-\mathcal{F}_0$ it suffices to prove  lemma for $\mathcal{F}_0$. 
These functions have $N+1$ components. Denote $\mathcal{F}^{i}$ and $\mathcal{F}^i_{0}$ the $i^{\text{th}}$ component of $\mathcal{F}$ 
and $\mathcal{F}_0$. Taking $(X,Y)$ and $(X',Y')$
both belonging  to some compact subset $\mathcal{K}\subset E\times F$, one gets for all $x\in \overline{\Omega}$ and $i=0,\cdots,N$, 
using the fact that $\mathcal{F}^i(\cdot,x)$ is locally Lipschitz and $\mathcal{F}^i(X+Y,\cdot)$ is smooth,
\begin{equation*}
|\mathcal{F}^{i}(x,X+Y(x))-\mathcal{F}^{i}(x,X'+Y'(x))|\leq C(\mathcal{K})\left(\|X-X'\|_E+\|Y-Y'\|_F\right)
\end{equation*}
where $C(\mathcal{K})$ is a positive constant depending on $\mathcal{K}$. Since $\mathcal{F}_0^i=\frac{1}{|\Omega|}\int_\Omega \mathcal{F}^i$, this yields 
\begin{equation*}
|\mathcal{F}_{0}^i(X,Y)-\mathcal{F}_{0}^i(X',Y')|\leq C(\mathcal{K})\left(\|X-X'\|_E+\|Y-Y'\|_F\right)
\end{equation*}
for all $i=0,\cdots,N$. It follows
\begin{equation*}
\|\mathcal{F}_{0}(X,Y)-\mathcal{F}_{0}(X',Y')\|_E\leq (N+1)C(\mathcal{K})\left(\|X-X'\|_E+\|Y-Y'\|_F\right)
\end{equation*}
which proves that $\mathcal{F}_0$ is $C^0$ from $E\times F$ into $E$ and so  is $\mathcal{G}_1$  from $E\times F$ into $F$.\\

Since $R\mapsto f_i(R,x)$ is assumed to be $C^1$ with  locally Lipschitz  derivative, the proof of the $C^1$ smoothness follows the same lines and we omit it.

\end{proof}

\noindent The theorem \ref{centralmanifold} also requires that $\mathcal{F}_0$ and $\mathcal{G}_1$ 
as well than their derivatives are   bounded independently on $\eps$. 
Obviously, this boundedness assumption does not hold in general. 
However, by  theorem \ref{chap5thmaj}, one already knows that every solution is bounded in $(C^0(\overline{\Omega}))^{N+1}$ independently on $\eps$ and $t$. 
It follows with the definition of the norm $E\times F$ that, for some large enough $M>0$, we have
$$\|X^\eps(t)\|_E+\|Y^\eps(t)\|_F\leq  M.$$ 
It then suffices to conveniently truncate $\mathcal{F}_0$ and $\mathcal{G}_1$
 outside the set $\{(X,Y)\in E\times F,\; \|X(t)\|_E+\|Y(t)\|_F\leq M\}$.\\

\noindent It follows that  Theorems \ref{centralmanifold} and \ref{errorbound} as well as propositions \ref{main1general} and \ref{main2general}
apply to the system $S_\eps^{sf}$ 
(defined in section \ref{sec22}). 
\subsection{Proof of the Theorem \ref{mains1} and \ref{mains2}}
Now, we apply the propositions \ref{main1general} and \ref{main2general} to the case of $S_\eps^{sf}$.\\ 
Since we are interested in biologically relevant solutions, we are only interested in nonnegative solutions 
which leads to some additional difficulties.
Let us start with the following lemma with is the key to deal with the positive quadrant near the boundaries.

\begin{lemma}\label{lemmaconservation} Let $\mathcal{M}^\eps=\{(X,h(X,\eps)),X\in E\}$ be a central manifold for $S_\eps^{sf}$ 
defined in Theorem \ref{centralmanifold}.
Denote $h(X,\eps)=(h_i(X,\eps))_{0\leq i\leq N}\in F$ and  $X=(r,u_1,\cdots,u_N)\in E.$\\
Then  there exists a function $g\in C^0 (E\times [0,1]; F)$ 
 such that for any $i=1,\cdots,N$ one has
$$h_i(X,\eps)=u_i g_i(X,\eps).$$
\end{lemma}
\begin{proof}
Since the nonnegative quadrant is invariant for $S_\eps$ 
and since $\mathcal{M}^\eps$ is invariant for $S_\eps^{sf}$, 
one sees that for any $X=(r,u_1,\cdots,u_N)\in \R_+^{N+1}$, one has for any $i=1,\cdots,N$ and $x\in \Omega$,
$u_i+h_i(X,\eps)(x)\geq 0.$
In particular, if $u_i=0$ it follows by the continuity of $h(\cdot,\eps)$ from $E$ to $F$
 that for all $x\in \Omega$, $h_i(X_{|u_i=0},\eps)(x)\geq 0$.\\
Besides, since $h(X,\eps)\in F$, one gets $\int_\Omega h_i(X_{|u_i=0},\eps)(x)dx=0$ and then $h_i(X_{|u_i=0},\eps)\equiv 0$.\\
Now, since $h(\cdot,\eps)\in C^1(E;F)$, one sees that 
$\frac{1}{u_i} h_i(X,\eps)$ converges in $F$ as $u_i\to 0$ and we are able to write 
$$h_i(X,\eps)=u_i g_i(X,\eps).$$
Since $h\in C^1(E\times [0,1]; F)$, the regularity of $g$ follows.
\end{proof}

The following lemma ensures that the stationary solutions of $S_\eps^{sf}$, constructed in the 
proposition \ref{main1general}, correspond to nonnegative stationary solutions of $S_\eps$.
\begin{lemma}\label{positivity}
Assume that the system $S_0^{c}$ admits a nonnegative hyperbolic stationary solution  denoted by 
$$p^0=(r^0,u_1^0,\cdots,u_N^0)\in\R_+^{N+1}.$$
Let $P^\eps(x)=(p^\eps,h(p^\eps,\eps)(x))$ be the stationary solution of $S_\eps^{sf}$ 
defined in the Theorem \ref{main1general}. The corresponding stationary solution of $S_\eps$ is denoted by
$$\bW^\eps(x)=p^\eps+h(p^\eps,\eps)(x):=(R^{\eps}(x),U_1^{\eps}(x),\cdots,U_N^{\eps}(x)).$$ 
Then for small enough $\eps>0$ one gets $R^\eps(x)>0$ for all $x\in \overline{\Omega}$ and 
$$u_i^0>0\Rightarrow U_i^\eps(x)> 0,\;\forall x\in\overline{\Omega}\;\text{ and } u_i^0=0\Rightarrow U_i^\eps \equiv 0.$$
\end{lemma}
\begin{proof}
Since $h(p^\eps,\eps)(x)=O(\eps)$, and $p^\eps \to p^0$ as $\eps \to 0$,  
if a component of $p^0$ is positive, so  is the corresponding component of $W^\eps(x)$  for small enough $\eps$.
 It is clear that $r^0>0$ and then $R^\eps(x)>0$ for all $x\in \overline{\Omega}$.  
Now, up to a rearrangement, suppose that $p^0=(r^0,u_1^0,\cdots,u_{N-1}^0,0)$. 
One knows that there exists a stationary solutions $W^\eps(x)=p^\eps+h(p^\eps,\eps)(x)$  of $S_\eps$. 
We now show that $U_N^\eps\equiv 0$. 
Thanks to the lemma \ref{lemmaconservation}, it suffices to show that $u_N^\eps:=p^\eps_N=0$.
Hence, define $\widetilde{S}^c_\eps$ the  subsystem without the species $N$ 
 similarly to the corresponding 
systems,  $S_\eps^{c}$.
Since $p^0$ is a hyperbolic stationary nonnegative solution of  $S_0^c$,
$\widetilde{p^0}=(r^0,u_1^0,\cdots,u_{N-1}^0)$ is a hyperbolic stationary non negative solution of $\widetilde{S}^c_0$. 
Lemma \ref{stationaryS0} applied to $\widetilde{S}_\eps^{c}$ allows to define  a stationary solution $\widetilde{p^\eps}$ of $\widetilde{\mathcal{S}}_\eps$. 
It follows that $(\widetilde{p^\eps},0)$ is a stationary solution of $S_\eps^{c}$  
and the uniqueness of $p^\eps$ in the neighborhood of $p^0$ yields to $p^\eps=(\widetilde{p^\eps},0)$, that is $u_N^\eps=0$ which end the proof.
\end{proof}\\
%
%
%
%
%
%
\begin{proof1} This theorem follows directly from the theorem \ref{stationaryagregated}
together with the proposition \ref{main1general} and  the lemma \ref{positivity}.
%
%
%
\end{proof1}\\


The proof of theorem \ref{mains2} uses strongly  proposition \ref{main2general}. 
The following lemma ensures that the assumption $(i)$ of this proposition
is satisfied.
\begin{lemma}\label{lemmarepuls}
Define the two subsets of $\R^{N+1}$:
$$\mathcal{Q}=\R_+^{N+1},\qquad \mathcal{Q}_1=\{(r,u_1,\ldots,u_N)\in \mathcal{Q},\; u_1>0\}$$
and, for  any positive scalar $\alpha$, define the two subsets of $\left( C^0(\overline{\Omega})\right)^{N+1}$
$$Q(\alpha):=\{(R,U_1,\ldots,U_N)\in C^0(\overline{\Omega}),\; \forall x\in \Omega,\; R(x)\geq\alpha\text{ and for each $i\geq 1$, } U_i(x)\geq 0\}$$
$$Q_1(\alpha):=\{(R,U_1,\ldots,U_N)\in Q(\alpha),\; \exists x\in \Omega,\;U_1(x) > 0\}.$$
For any initial data $\bW(0):=(R(0),U_{1}(0),\ldots, U_{N}(0))\in Q(\alpha)$, one notes $\Pi_E\bW(0)=(r(0),u_1(0),\cdots,u_N(0))$, $Z_0=\left(\Pi_E \bW(0),\Pi_F \bW(0)\right)\in E\times F$ and
 $X_0^\eps(Z_0)=(r^\eps(0),u_1^\eps(0),\cdots,u_N^\eps(0))$ the modified initial data defined in the theorem 
\ref{errorbound}-(iii). \\
For any $\alpha$, there exists $\eps(\alpha)>0$ such that for each $\eps\in(0,\eps(\alpha))$, the following holds true.

\begin{enumerate}[(i)]
 \item For any initial data $\bW(0)\in Q(\alpha)$ one gets $X_0^\eps(Z_0)\in \mathcal{Q}$.
 \item Assume that $r_1^*<r_j^*$ for any $j\neq 1$. Then, for any initial data $\bW(0)\in Q_1(\alpha)$ one gets $X_0^\eps(Z_0)\in \mathcal{Q}_1$.
\end{enumerate}

\end{lemma}
\begin{proof}  Let $\alpha>0$ be fixed and take $\bW(0)\in Q(\alpha)$. 
From $\|\Pi_E(\bW(0))-X_0^\eps(Z_0)\|_E\leq C\eps$, we deduce  $r^\eps(0)=r(0)+O(\eps)\geq \alpha+O(\eps)>0$ provided $\eps$ is small enough. Moreover, 
the conservation of the line $U_i\equiv0$ by both the system $S_\eps$ and $S_0^c$ implies $u_i^\eps(0)\geq 0$ (see the remark \ref{remarkx0})
which proves the point $(i)$.\\

The only difficulty in proving $(ii)$  is that, {\it a priori}, 
taking an initial data $\bW(0)\in Q_1(\alpha)$ can  provide a modified initial data
$X_0^\eps(Z_0)\notin \mathcal{Q}_1$, i.e. such that $u_1^\eps(0)=0$. 
We  show that this can not hold by contradiction\footnote{
Let us remarks at this step that one gets 
$u_1(0):=\frac{1}{\Omega}\int_\Omega U_1^\eps (0,x)dx+O(\eps)$ so that for any initial data
$U_1(x,0)>0$, one gets $u_1^\eps(0) >0$ for small enough $\eps$ {\it depending on } $\bW(0)$. 
It follows directly that the global asymptotic behavior holds true when $\bU_1^\eps(0)$ is far enough from the boundary.
 One can also reformulate this by saying that for any compact subset $\mathcal{K}$ of $Q_1(\alpha)$, 
 there exists $\eps(\mathcal{K})$ such that for any $\eps\in(0,\eps(\mathcal{K}))$, the global asymptotic behaviors holds. \\
The only problem occurs when $U_1=O(\eps)$ which can very hold in $Q_1(\alpha)$.
}.\\
Assume that  $\bW(0)\in Q_1(\alpha)$ and that 
$X_0^\eps(Z_0)$ verifies $u_1^\eps(0)=0$.
Denote $X^{\eps,[0]}(t):=(r^{\eps,0}(t),u_1^{\eps,0}(t),\cdot,u_N^{\eps,0}(t))$  the solution of $S_0^c$ 
with $X^{\eps,[0]}(0)=X_0^\eps(Z_0)$.\\
The line $u_1=0$ being invariant for $S_0^c$, one has
\begin{equation}\label{u1eps0eq0}
\forall t\geq0,\quad u_1^{\eps,[0]}(t)=0\end{equation}
and then, by the proposition \ref{hyp0},
\begin{equation}\label{reps0torchap}
 \lim_{t\to+\infty} r^{\eps,[0]}(t)= \widehat{r}\text{ where }\widehat{r}=r_k^*\text{ for some }k\neq 1.
 \end{equation}
 
 Now, let $X^{\eps,[\infty]}(t)=(r^{\eps,\infty}(t),u_1^{\eps,\infty}(t),\ldots,u_N^{\eps,\infty}(t))$ be a solution of $S_\eps^c$ whis initial data
 $X^{\eps,[\infty]}(0)$ given by  Theorem \ref{errorbound}-(ii). We claim that $u_1^{\eps,[\infty]}=0$. 
 Indeed, from \eqref{u1eps0eq0} and \eqref{reps0torchap}, this theorem implies
\begin{equation}\label{poseq1}
\forall t>0,\;0\leq u_1^{\eps,\infty}(t)\leq C\eps
\end{equation}
and for $t$ large enough,
\begin{equation}\label{poseq2}
|r^{\eps,\infty}(t)-\widehat{r}|\leq C\eps.
\end{equation}
Thus, by  lemma \ref{lemmaconservation}, and smoothness of $h$ and $f_1$, one gets for $t$ large enough
$$\frac{d}{dt} u_1^{\eps,\infty}= u_1^{\eps,\infty}\left(\widetilde{f_1}(\widehat{r})-\widetilde{m_1}+O(\eps)\right)(1+O(\eps)).$$
Moreover, since $r_1^*<r_k^*$ for all $k\neq 1$, one has $\widehat{r}>r_1^*$ and then, 
if $\eps$ is small enough (depending only on the gap $\widetilde{f_1}(\widehat{r})-\widetilde{m_1}$), one gets 
$$\left(\widetilde{f_1}(\widehat{r})-\widetilde{m_1}+O(\eps)\right)>0.$$
\noindent It follows that if $u_1^{\eps,\infty}(0)>0,$ then 
${\displaystyle \lim_{t\to +\infty}u_{1}^{\eps,\infty}(t)= +\infty}$ a contradiction with \eqref{poseq1}.\\
Thus $u_1^{\eps,\infty}(0)=0$ and then $u_1^{\eps,\infty}(t)=0$ for all $t\geq 0$. 
It follows by the Theorem \ref{errorbound} that, for some positive constants $C$ and $\mu'$,
\begin{equation}\label{Uto0}
\|U_1^\eps(\cdot,t)\|_\infty\leq C e^{-\mu'\frac{t}{\eps}},
\end{equation}
and from \eqref{reps0torchap} we deduce for large enough $t>0$,
\begin{equation}
\|R^{\eps}(\cdot,t)- \widehat{r}\|_\infty\leq C\eps.
\end{equation}
On the other hand, for any $t>0$ and $x\in\Omega$, the real (component of the ) solution $U_1^\eps(x,t)$ is positive and verifies 
\begin{equation}\label{eqU}
\begin{split}
\partial_t U_1^\eps(x,t)&=U_1^\eps(x,t)\left(f_1(x,R^{\eps}(x,t))-m_1(x)\right)+\frac{1}{\eps} A_i U_1^\eps(x,t)\\
						&=U_1^\eps(x,t)\left(f_1(x,\widehat{r})-m_1(x)+O(\eps)\right)+\frac{1}{\eps} A_i U_1^\eps(x,t)\\
\end{split}
\end{equation}
It is well known that the operator $\left(f_1(x,\widehat{r})-m_1(x)\right)+\frac{1}{\eps} A_i $ has a principal eigenvalue $\lambda_\eps$ and 
a corresponding function $\phi_\eps>0$. Moreover (see for instance \cite{hutson1}) 
$\lambda_\eps$ tends continuously to $ \widetilde{f}_1(\widehat{r})-\widetilde{m_1}>0$ as $\eps\to 0$. 
Multiplying \eqref{eqU} by $\phi_\eps$ and integrating over $\Omega$, one obtains for large enough $t>0$,
$$\partial_t \int_\Omega U_1^\eps(t,x)\phi_\eps(x)dx=(\lambda_\eps+O(\eps))\int_\Omega U_1^\eps(t,x)\phi_\eps(x)dx.$$
If $\eps$ is small enough (depending only on $\widetilde{f}_1(\widehat{r})-\widetilde{m_1}$), it follows that 
$t\mapsto \int_\Omega U_1^\eps(t,x)\phi_\eps(x)dx$ is a positive increasing function for $t$ large enough which contradicts \eqref{Uto0}. 
It follows that $u_1^\eps(0)>0$ and the point $(ii)$ is proved.
\end{proof}\\

\begin{proof2}
Take $\bW(.,0)\in Q$. By the theorem \ref{chap5thmaj},  one has for some constant $M>0$
\begin{equation}\label{comparison}
 \partial_t R^\eps(x,t)-\frac{1}{\eps}A_0 R^\eps(x,t)\geq I(x)-m_0(x) R^\eps(x,t)-M\sum_{i=1}^N f_i(x,R^\eps(x,t),\; t>0,\;x\in\Omega.
\end{equation}

\noindent The comparison principle in parabolic equations shows that $R^\eps(x,t)>\underline{R}(x,t)$ where $\underline{R}(x,t)$ is a solution of \eqref{comparison}
with an equality, together with zero flux boundary conditions and the initial values $\underline{R}(x,0)=R(x,0)$. A lower-upper solution method shows that
$\underline{R}(x,t)\to \Phi(x)$ as $t\to+\infty$ where $\Phi(x)$ is the only stationary solution of \eqref{comparison} (with equality). 
From $I\not\equiv 0$ and the strong maximum principle, we deduce
$\Phi(x)>0$ for all $x\in\overline{\Omega}$. As a consequences, there exists a scalar $0<\alpha<\min_{x\in\overline{\Omega}} \Phi(x)$
and a time $t_0\geq 0$ such that
$R^\eps(t_0,x)>\alpha$ for any $t>t_0$. Since $S_\eps$ conserve the positive quadrant, it follows that $W(\cdot,t)\in Q(\alpha)$ for $t\geq t_0$
Hence, without loss of generality, one may assume that $\bW(\cdot,0)\in Q(\alpha)$ resp. $Q_1(\alpha)$). It follows from  lemma \ref{lemmarepuls} that all
the perturbed initial data appearing in  theorem \ref{errorbound} lies on $\mathcal{Q}:=\R_+^{N+1}$ (resp. $\mathcal{Q}_1$). 
Now,   by the proposition \ref{main2general},  the points $(ii)$ and $(iii)$ 
of the theorem follow from the points $(ii)$ and $(iii)$ of the proposition \ref{hyp0}. It remains to prove the point $(i)$.\\

 Let $i\in\{1,\cdots,N\}$. First, it is well known (and easy to check) that for any  initial data $X(0)\in\mathcal{Q}$, one gets
$\limsup_{t} r(t)\leq r_0^*$. 
Arguing as in the proof of the lemma \ref{lemmarepuls}, one deduces that for large enough $t$,
$$\partial_t u_i^{\eps,[\infty]}(t)\leq u_i^{\eps,[\infty]}(t)(\widetilde{f_i}(r_0^*)-\widetilde{m_i}+O(\eps))(1+O(\eps)).$$
The inequality $r_0^*<r_i^*$ reads exactly $\widetilde{f_i}(r_0^*)-\widetilde{m_i}<0$. It follows that $u_i^{\eps,[\infty]}(t)\to 0$ for small enough $\eps$ and any initial data $X_0^\eps\in \mathcal{Q}$. By virtue of the theorem \ref{errorbound}, for  some initial data $X_0^\eps\in \mathcal{Q}$, one has
 $$\|U_i^\eps(\cdot,t)-u_i^{\eps,[\infty]}(t)\|_\infty \leq C e^{-\mu'\frac{t}{\eps}}$$ and $\|U_i^\eps(\cdot,t)\|_\infty \to 0$ follows.

\end{proof2}

\section{The best competitor in average}\label{application}
Rougly speaking, the  Theorem \ref{mains2} may be summarized as follow. If the diffusion rate is large enough, then the CEP holds for the system $S_\eps$. 
At most one species survives namely {\it the best competitors in average}, 
that is the species associated with the smallest $r_i^*$. 
Inversely, looking at the system $S_\eps$ without diffusion, one defines for each $x\in \Omega$, $R_0^*(x)=I(x)/m_0(x)$ and  $R_i^*(x)$ 
the only solution of $f_i(R_i(x),x)=m_i(x)$ if it exists and $R_i^*(x)=+\infty$ else. 
We say that the $i^\text{th}$ species is a strong local competitor if there exists $x\in \Omega$ 
such that $R_i^*(x)<R_j^*(x)$ for  all $j\neq i$. 
We say that the $i^\text{th}$ species is {\it a weak local competitor} if for all $x\in \Omega$, there exists $j$ such that $R_i^*(x)>R_j^*(x)$. 
A weak local competitor can not survive to the competition without 
diffusion\footnote{ Numerical evidence show that a weak local competitor can no survive to the competition 
for small enough diffusion rates. As it is proved in \cite{DUMA}, a rigourous studied of  {\it stationnary} 
solutions for small diffusion supporte these evidences.}.

This has two implications.
Fistly, this highlights that different competitive strategies may be selected depending if the environment is well-mixed or not.
Secondly, this indicates that for  {\it intermediate} diffusion rates, several competitive strategies may  yield   {\it coexistence}.

Thus, the below detailled phenomena  are indicators of the possibility of a given environement to promote coexistence by mixing both the {\it local} aspects and
the {\it global} ones. This type of local/global duality  has been discussed within a different framework in \cite{Chesson2000} for instance.

We now discuss on precise examples three phenomena showing that the best competitor in average
can be a weak local competitor.\\

\noindent For a given function $g\in C^0(\Omega)$, (resp. a vector $g$ if $\Omega$ is finite), denote E the average of $g$.
The number $r_i^*$ (defined in figure \ref{figureri}) reads
$$r_i^*=E(R_i^*)+ J_i+H_i$$
wherein we have set
$$J_i=\widetilde{f_i}^{-1}\big(E(m_i)\big)-E\big(\widetilde{f_i}^{-1}((m_i))\big)\text{ and }H_i=E\big(\widetilde{f_i}^{-1}(m_i)\big)-E\big(f_i^{-1}(m_i)\big).$$

\noindent The biological interpretation of each term is as follows.
\begin{itemize}
\item {\bf The (averaged) local competitive strength} is represented by $E(R_i^*)$.
 The stronger local competitor the species $i$ is, the  smaller is $E(R_i^*)$.\\
 This phenomena is of particular interest in a {\it three species (or more) situation} since a generalist 
 (a species which is a weak local competitor but with a small  $E(R_i^*)$) may 
 lose the competition on each patche but win the competition in average.\\
 From a coexistence point of view, this permits to several (three or more) species  to coexiste for an intermediate diffusion rate, 
 while they can not coexist neither for a small nor a large diffusion rate.
\item {\bf The non linear effect} is represented by $J_i$. This term is null if either $\widetilde{f_i}$ is linear or $m_i$ is constant.
 Usually, the consumption function $\widetilde{f_i}$ is increasing and concave so that $\widetilde{f_i}^{-1}$ is convex. 
In this case, due to the Jensen inequality, $J_i$ is negative.\\
Hence, {\it the nonlinear effect improves the competition strength of species.} \\
From a coexistence point of view and for intermediate diffusion rate,
this is the phenomena which permits coexistence in the classical {unstirred} chemostat \cite{walt1, Wu1} or in the classical {gradostat} \cite{waltbook}.
\item {\bf The heterogeneous effect of the consumption} is represented by $H_i$.
Basically, it represents the effect of the heterogeneity of the consumption function $f_i(x,\cdot)$ and it is null if $f_i=\widetilde{f_i}$.\\
{ The larger the consumption $f_i(j,\cdot )$ is at location $j\in\Omega$ where $R_i^*(j)$ is large, the smaller is  $H_i$.} \\
Hence, {\it a fast dynamics on the sites where $R_i^*(j)$ is small improves the averaged competitive strenght of the species.}\\
From a coexistence point of view and for intermediate diffusion rate,
this  phenomena increase the possibility of coexistence in the generalised chemostat (or gradostat), see \cite{CasMad}.

\end{itemize}

Now, we illustrate this three phenomena on examples. To simplify the  discution, we focus here on the case of a two patches model: $\Omega=\{1,2\}$ and $A_i\in \R^{2\times 2}$ defined for each $i$ as
$A=A_i=\left[\begin{array}{cc} -1&1\\1&-1\end{array}\right].$
Besides, we assume that $R_i^*(j)$ is well defined for all $j=1,2$. Here, for $g=(g(1),g(2))$, 
one has $E(g)=\frac{1}{2}(g(1)+g(2))$.
\subsection{ The local competitive strength}
Define the special case of $S_\eps$ (in $\Omega=\{1,2\}$) for three species (with positive initial data)
\begin{equation}\label{localcompet}
\left\{\begin{array}{l}
d_t R(j,t)=1- R(j,t)-\sum{ U_i(j,t)R(j,t)}+\frac{1}{\eps} (A R)(j,t),\\
d_t U_i(j,t)=( R(j,t)-m_i(j))U_i(j,t)+\frac{1}{\eps} (A U_i)(j,t),\quad i=1,2, 3
\end{array}\quad j=1,2\right.\end{equation}

For $j=\{1,2\}$, one gets $R_0^*(j)=1$ and  $R_i^*(j)=m_i(j)$ for $i=1,2,3$. We also assume that $1>R_i^*(j)$ for $i=1,2,3$ and $j=1,2$.
Here, $r_i^*$ reads $r_i^*=\frac{1}{2} (m_i(1)+m_i(2))$.

One claims that it is possible to  find  three vector $m_i$ such that $R_3^*(j)>min(R_1^*(j),R_2^*(j))$. 
for all $j\in \{1,2\}$ and $r_3^*<\min(r_1^*,r_2^*)$.
It suffices to choose $m_i$ such that for instance $m_1(1)<m_3(1)<m_2(1)$ and $m_2(2)<m_3(2)<m_1(2)$ and $m_3(1)+m_3(2)<m_i(1)+m_i(2)$ for $i=1,2$.
The vectors $m_1={}^t(0.1,0.9)$, $m_2={}^t(0.9,0.1)$ and $m_3={}^t(0.4,0.4)$ suit.\\
{ Biologically, an interpretation is  that the first and second species are specialists (the best competitor on one site and the weakest on the other site)
whereas the third species is a generalist (a weak competitor but the weakest on no site)}

Hence, according to Theorem \ref{mains2}, the third species is the best averaged competitor but a weak local competitor.
Hence, {\bf without migration the first species survives on the site $1$, the second on the site $2$ and the third nowhere, 
while for fast migration, the two first species do not survive and the third is the only survivor.}

For an intermediate diffusion rate, we guess that the 
three species may eventually coexiste (even if it means increasing the number of sites\footnote{As it is shown in \cite{HofSo},  stationary coexistence of $N$ species 
in $P$ sites is generically impossible. Thus, $3$ species can not coexist in less than $3$ patches.}).

in the sense that the species 1 and 2 are good only on the sites 1 and 2 respectivly, while the species 3  \\

\subsection{ The non linear effect}
Here we assume that the consumption function $f_i$ is homogeneous so that $H_i=0$. One discuss the particular cases of Holling type II functions :
$f_i(R)=\frac{R}{k_i+R}$.
The nonlinear effect is more important if the function $f_i$ is very nonlinear. For Holling type II functions,
this can be measured by the number $k_i$:
\begin{equation}\label{Jensen}
J_i=k_i\left[\frac{E(m_i)}{1-E(m_i)}-E\left(\frac{m_i}{1-m_i}\right)\right].
\end{equation}
Due to the Jensen's inequality, $J_i$ is non positive and is null if and only if $m_i$ is constant. 


As a consequence one can constructed an explicit example of two species competing for the same resource $R$ such that the species 1 is the
best local competitor on each site  while the second species is the best competitor in average. 
An explicit example is the following 

\begin{equation}
\left\{\begin{array}{l}
d_t R(j,t)=10- R(j,t)- \frac{R(j,t)}{1+R(j,t)}U_1(j,t)-\frac{R(j,t)}{0.25+R(j,t)}U_2(j,t)+\frac{1}{\eps} (A R)(j,t),\\
d_t U_1(j,t)=( \frac{R(j,t)}{1+R(j,t)}-m_i(j))U_1(j,t)+\frac{1}{\eps} (A U_1)(j,t),\\
d_t U_2(j,t)=( \frac{R(j,t)}{0.25+R(j,t)}-m_i(j))U_2(j,t)+\frac{1}{\eps} (A U_2)(j,t),\\
\end{array}\right.\end{equation}
where $m_1={}^t(0.38,34/41)$ and $m_2={}^t(0.75,20/21)$. Explicite computations give
$R_1^*={}^t(0.6129   , 4.8571)$ and $R_2^*={}^t(0.75  ,  5)$ while $r_1^*\approx 1.5293$ and $r_2^*=1.43$.\\

\noindent As a consequence, 
{\bf the first species is the only survivor for slow migration will the species 2 will be the only survivor for fast enough migration.}\\
One can also build an example of a single species and we obtain:
{\bf due to the nonlinear effect, a species which is  able to survives on no site without migration can survive for fast enough migration.}\\
\subsection{The heterogeneous effect of the consumption}
In the previous discussion, the heterogeneity take place only on the mortality. 
If the consumption function itself is heterogeneous, a third phenomenon occurs. 
Here, we discuss the case of a linear consumption function so that $J_i$ is null. Let take
$$f_i(j,R)=C_i(j) R.$$

We illustrate this phenomena 
 on the following  two species system 
 \begin{equation}\label{localcompetcov}
\left\{\begin{array}{l}
d_t R(j,t)=1- R(j,t)-\sum{ C_i(j)U_i(j,t)R(j,t)}+\frac{1}{\eps} (A R)(j,t),\\
d_t U_i(j,t)=( C_i(j)R(j,t)-m_i(j))U_i(j,t)+\frac{1}{\eps} (A U_i)(j,t),\quad i=1,2
\end{array}\right.\end{equation}
We will see that {\it the best competitor in average can be the weakest competitor everywhere} in that case.\\
This phenomena is similar to the Fitness-density covariance in heterogeneous environment stress by Chesson {\it et al.} \cite{Chesson2000}. 
Indeed, noting  $cov(f,g)=E(fg)-E(f)E(g)$, one get $r_i^*=E(R_i^*)+cov(\frac{C_i}{E(C_i)},R_i^*)$ and $r_0^*=E(R_0^*)+cov(\frac{m_0}{E(m_0)},R_0^*)$.\\

A species may be the weakest local competitor and the best competitor in average. Indeed, 
 $r_1^*>r_2^*$ if and only if $cov(\frac{C_2}{E(C_2)},R_2^*)-cov(\frac{C_1}{E(C_1)},R_1^*)<E(R_1^*)-E(R_2^*)$ which implies
 \begin{proposition}(Competitive covariance in heterogeneous environment)\\
 $r_2^*<r_1^*$ if and only if   $cov(\frac{C_2}{E(C_2)},R_2^*)-cov(\frac{C_1}{E(C_1)},R_1^*)<E(R_1^*)-E(R_2^*)$.
 In particular, one may have $R_1^*(j)<R_2^*(j)$ for each $j\in \Omega$.
 \end{proposition}
If $R_1^*(j)<R_2^*(j)$ for each $j\in \Omega$, then it is necessary that $cov(c_2,R_2^*)-cov(c_1,R_1^*)$ is negative and small enough. 
This means that either the best local competitor as maximal consumption rate on  bad site ( where $R_1^*(j)$ is large), 
or the weak local competitor has maximal consumption rate on good site,(where $R_2^*(j)$ is small).\\
The following result give a necessary and sufficient condition on $R_i^*$ for this phenomena may happen.
\begin{proposition}
Suppose that the first species  is the best competitors everywhere, that is $R_1^*(j)<R_2^*(j)$ for all $j\in \Omega$.\\
If $\max_{j\in \Omega} R_2^*(j)<\min_{j\in \Omega} R_1^*(j)$, then there exists two smooth positive vectors  $C_1$ and $C_2$ such that $r_1^*>r_2^*$ 
\end{proposition} 
\begin{proof}
$R_1^*(j)$ and $R_2^*(j)$ being fixed, one gets $r_i^*=\frac{C_i(j)R_i^*(j)}{C_i(j)}$.
It suffices to find two vectors such that $E(C_1 R_1^*) E C_2)<E( C_2 R_2^*)E(C_1)$. 
Denoting $j_1$ and $j_2$ such that $\min_{j\in \Omega} R_1^*(j)=R^*_1(j_1)$ and $\max_{j\in \Omega} R_2^*(j)=R_2^*(j_2)$,
it suffices to choose two vectors, such that for $i=1,2$, $C_i(j)\approx \delta(j=j_i)$. It  comes $r_i^*\approx R^*_i(j_i)$ which end the proof.
\end{proof}

According to the Theorem \ref{mains2}, the second species is the only survivor if $\eps$ is small enough. 
Numerical simulations indicate that, as expected,  the first species is the only survivor for large $\eps$, 
the second is the only survivor for small $\eps$, and the two species coexist for an intermediate value of $\eps$. 
Similiar arguments on single species models show that a species may not survive locally but survive globaly or conversely. In conclusion
{\bf a fast dynamics on good sites increases the averaged competitive strenght of a species.}\\
This underline the importance of the spatial heterogeneity together with the value of the diffusion rates on coexistence phenomena.\\

\section{Conclusion}
\paragraph{}
In this text, we have studied a system of $N$ species competing for a single resource where populations and resource depend both on time and space. 
The demography is described {\it at each site} by a chemostat model, assuming increasing consumption functions and constant yields. 
The diffusions are assumed {\it fast} which induces an average effect on the spatial repartition of the populations.
Our results are as follows. 
\paragraph{}
We show that the dynamics is asymptotically well described, up to an exponentially small error term, by a system involving $N+1$ equations instead of 
$N+1$ equations {\it per site}, describing the dynamics of the total number of individual. 
In turn, this reduced system is well described, up to an order one small error term, by a standard homogeneous chemostat system, 
called the aggregated system, which can be {\it explicitly computed}.
\paragraph{}
The main result of this work is that, if the aggregated system verifies the CEP, then the original system verifies the CEP, {\it for fast enough diffusions}.\\
This result give a justification to ''well-mixed'' assumption done in the statement of homogeneous chemostat models. 
Besides, the parameters of the aggregated system can be explicitly computed.
\paragraph{}
In particular, we show that the only survivor is the best competitor {\it in average}. Moreover, we note that the best competitor in average can be
{\it the best competitor nowhere}, and indeed, if the heterogeneity concern both the mortalities and the consumption functions or if the consumptions 
function are non linear, the best competitor 
in average can be {\it the weakest competitor everywhere} (see section \ref{application} for a definition of weak/best competitors).\\
Moreover, these results give indication about the possibility that a heterogeneous environment promotes coexistence for intermediate diffusion rates.
Note that all the results of this work hold for a  gradostat model, replacing the continuous space $\Omega$ by a finite number of sites, and the 
diffusion operators by a migration matrix  assuming to be irreducible. In that case, the Perron-Frobenius Theorem give all the spectral information 
and the central manifold Theorem state in \cite{Castella2009} apply directly leading to the similar results.\\

Several ways of future investigation can extend this study.\\
First,  Theorems \ref{mains1} and \ref{mains2} assume that the stationary solution of the aggregated problem are hyperbolic, 
that is the numbers $r_i^*$ are different. In an homogeneous chemostat (together with some additional assumption), 
the global dynamics can be described even if the $r_i^*$ are equal. The global attractor is then a family of non isolated stationary solutions 
instead of a unique stationary solution, and several species can survive \cite{waltbook}. The Theorem \ref{errorbound} gives directly some informations 
on the dynamics of the original system $S_\eps$, up to an error in $\eps$. However, the stationary solution being degenerate, the local inversion 
Theorem can no longer apply, and the  construction of  section 4 fails to describe completely the dynamics of the original system.   
In order to study more precisely this case, we have to calculate the reduced system at a higher order (up to an order 2 error term). 
This new system is still a system of $N+1$ differential equations, but with 
additional terms of order $\eps$. The dynamics of this systems is not known to our knowledge. Such a study can give several information 
of the ways the coexistence can happens and even on the way large diffusion leads to exclusion. 
\paragraph{}
Secondly, our study is restricted to the case of increasing consumption functions and constant yields. These assumptions are indeed used 
only from the the section 4. Various results are known in the case of an homogeneous chemostat with non monotone consumption functions
\cite{LI,WL,WX} or variable yields \cite{SW,SaMa}. An aggregated system can be compute for such case and determined which of this results can be applies.
\paragraph{}

\appendix

\end{document}